\documentclass[10pt,a4paper,final]{article}
\pdfoutput=1
\usepackage[latin1]{inputenc}
\usepackage{amsmath}
\usepackage{amssymb}
\usepackage{graphicx}
\usepackage{hyperref}
\usepackage{enumerate}
\usepackage{amsthm}
\usepackage{cite}
\usepackage{mathrsfs}
\usepackage{bbm}
\usepackage{subfigure}
\usepackage{float}
\usepackage{epsfig}
\usepackage[numbers]{natbib}
\usepackage{mathtools}
\usepackage{listings}
\usepackage{pgf}
\usepackage{listings}
\usepackage{epstopdf}
\usepackage[notcite,notref]{showkeys}
\usepackage[hmargin=3.5cm, vmargin=3.5cm]{geometry}
\usepackage{stmaryrd}
\usepackage{caption}

\theoremstyle{plain} 
\newtheorem{thm}{Theorem}[section] 
\newtheorem{prop}[thm]{Proposition}
\newtheorem{cor}[thm]{Corollary}
\newtheorem{lem}[thm]{Lemma}

\theoremstyle{definition} 
\newtheorem{defn}{Definition}[section]

\theoremstyle{remark}

\renewcommand{\P}{\mathbb{P}}
\newcommand{\Q}{\mathbb{Q}}
\newcommand{\E}{\mathbb{E}}
\newcommand{\F}{\mathcal{F}}
\newcommand{\ind}{\mathbbm{1}}

\newcommand{\N}{\mathbb{N}}
\renewcommand{\d}{\,\mathrm{d}}

\DeclareMathOperator{\Bin}{Bin}

\title{Unusually large components in near-critical Erd\H{o}s-R\'enyi graphs via ballot theorems}
\author{Umberto De Ambroggio\thanks{University of Bath, Department of Mathematical Sciences, Bath BA2 7AY, UK. \texttt{umbidea@gmail.com}}\hspace{1.5mm} and Matthew I.~Roberts\thanks{University of Bath, Department of Mathematical Sciences, Bath BA2 7AY, UK. \texttt{mattiroberts@gmail.com}}}
\begin{document}

\maketitle

\begin{abstract}
	We consider the near-critical Erd\H{o}s-R\'enyi random graph $G(n,p)$ and provide a new probabilistic proof of the fact that, when $p$ is of the form $p=p(n)=1/n+\lambda/n^{4/3}$ and $A$ is large,
	\[\mathbb{P}(|\mathcal{C}_{\max}|>An^{2/3})\asymp A^{-3/2}e^{-\frac{A^3}{8}+\frac{\lambda A^2}{2}-\frac{\lambda^2A}{2}}\]
	where $\mathcal{C}_{\max}$ is the largest connected component of the graph. Our result allows $A$ and $\lambda$ to depend on $n$. While this result is already known, our proof relies only on conceptual and adaptable tools such as ballot theorems, whereas the existing proof relies on a combinatorial formula specific to Erd\H{o}s-R\'enyi graphs, together with analytic estimates.
\end{abstract}

\section{Introduction}

The Erd\H{o}s-R\'enyi random graph, denoted by $G(n,p)$, is obtained from the complete graph with vertex set $[n]$ by independently retaining each edge with probability $p\in [0,1]$ and deleting it with probability $1-p$. We are interested in the size of the largest connected component $\mathcal C_{\max}$, or a typical connected component $C(v)$ for $v\in[n]$. It is well known (see e.g \cite{bollobas_book}, \cite{remco:random_graphs} or \cite{janson_et_al:random_graphs} for more details) that, if $p=p(n)=\gamma/n$ for constant $\gamma$, then $G(n,p)$ undergoes a phase transition as $\gamma$ passes 1:
\vskip0.2cm
\begin{enumerate}[(i)]
	\item if $\gamma < 1$ (the \textit{subcritical} case), then $|\mathcal C_{\max}|$ is of order $\log n$;
	\item if $\gamma=1$ (the \textit{critical} case), then $|\mathcal C_{\max}|$ is of order $n^{2/3}$;\label{critbasic}
	\item if $\gamma > 1$ (the \textit{supercritical} case), then $|\mathcal C_{\max}|$ is of order $n$.
\end{enumerate}
Motivated by the lack of a simple proof of (\ref{critbasic}), Nachmias and Peres \cite{nachmias_peres:CRG_mgs} used a martingale argument to prove that for any $n>1000$ and $A>8$,
		\[\mathbb{P}(|C(v)|>An^{2/3})\leq 4n^{-1/3}\exp\{-A^2(A-4)/32\}\]
	and
		\[\mathbb{P}(|\mathcal{C}_{\max}|>An^{2/3})\leq \frac{4}{A}\exp\{-A^2(A-4)/32\}.\]
They also gave bounds when $p = \frac{1+\lambda n^{-1/3}}{n}$ for fixed $\lambda\in\mathbb{R}$. The best known bound on the latter quantity is due originally to Pittel \cite{pittel:largest_cpt_rg} who showed that for $p$ of this form,
\[\lim_{n\to\infty}A^{3/2}e^{\frac{A^3}{8}-\frac{\lambda A^2}{2}+\frac{\lambda^2A}{2}}\mathbb{P}(|\mathcal C_{\max}|> An^{2/3} )\]
converges as $A\to\infty$ to a specific constant, which is stated to be $(2\pi)^{-1/2}$ but should be $(8/9\pi)^{1/2}$ due to a small oversight in the proof. More details, and a stronger result that allows $A$ and $\lambda$ to depend on $n$, are available in \cite{roberts:component_ER}. Both \cite{pittel:largest_cpt_rg} and \cite{roberts:component_ER} rely on a combinatorial formula for the expected number of components with exactly $k$ vertices and $k+\ell$ edges, which is specific to Erd\H{o}s-R\'enyi graphs and appears difficult to adapt to other models, together with analytic approximations.

We provide a new proof of asymptotics for $\mathbb{P}(|C(v)|>An^{2/3})$ and $\mathbb{P}(|\mathcal{C}_{\max}|>An^{2/3})$ that combines the strengths of the results mentioned above:
\begin{itemize}
\item it gives accurate bounds for large $A$ as $n\to\infty$;
\item it allows $A$ and $\lambda$ to depend on $n$;
\item it uses only robust probabilistic tools and therefore has the potential to be adapted to other models of random graphs.
\end{itemize}
This is the purpose of our main theorem, which we now state.

\begin{thm}\label{mainthm}
	There exists $A_0>0$ such that if $A=A(n)$ satisfies $A_0\le A = o(n^{1/30})$ and $p=p(n)=1/n+\lambda/n^{4/3}$ with $\lambda = \lambda(n)$ such that $|\lambda|\leq A/3$, then for sufficiently large $n$ and any vertex $v\in[n]$, we have 
		\[\text{(a)}\hspace{10mm} \frac{c_1}{A^{1/2}n^{1/3}}e^{-\frac{A^3}{8}+\frac{\lambda A^2}{2}-\frac{\lambda^2A}{2}} \leq \mathbb{P}(|C(v)|>  An^{2/3} ) \leq \frac{c_2}{A^{1/2} n^{1/3}}e^{-\frac{A^3}{8}+\frac{\lambda A^2}{2}-\frac{\lambda^2A}{2}}\]
		and
		\[\text{(b)}\hspace{10mm} \frac{c_1}{A^{3/2}}e^{-\frac{A^3}{8}+\frac{\lambda A^2}{2}-\frac{\lambda^2A}{2}} \leq \mathbb{P}(|\mathcal{C}_{\max}|> An^{2/3} )\leq \frac{c_2}{A^{3/2}}e^{-\frac{A^3}{8}+\frac{\lambda A^2}{2}-\frac{\lambda^2A}{2}}\]
	for some constants $0<c_1\le c_2<\infty$.
	\end{thm} 

\noindent	
Although our methods are not accurate enough to give the correct constant factor in the asymptotic, identified in \cite{roberts:component_ER}, we believe that the substantially more robust approach is worth the small sacrifice in precision. Indeed, the probabilistic arguments in \cite{nachmias_peres:CRG_mgs} have been adapted to critical random $d$-regular graphs by Nachmias and Peres \cite{nachmias:critical_perco_rand_regular}, the configuration model with bounded degrees by Riordan \cite{riordan:phase_transition_config}, and more recently a particular model of inhomogeneous random graphs by the first author and Pachon \cite{de_ambroggio_pachon:upper_bounds_inhom_RGs}. 

We remark that our proofs of the upper bounds in Theorem \ref{mainthm} are particularly straightforward, perhaps even more so than those in \cite{nachmias_peres:CRG_mgs}, despite giving a much more accurate bound. A key part of the argument will be the following simple \textit{ballot-type} result, which may be of independent interest.
\begin{lem}\label{ballotlemma}
	Fix $n\in\N$. Let $X_1,\dots,X_n$ be $\mathbb{Z}-$valued random variables and suppose that the law of $(X_1,\dots,X_n)$ is invariant under rotations (it may depend on $n$). Define $S_0 = 0$ and $S_t = \sum_{i=1}^{t}X_i$, for $t\in [n]$. Then for any $j\in\N$,
	\begin{equation*}
	\mathbb{P}(S_t>0\hspace{0.15cm} \forall t\in [n],\,S_n=j)\leq \frac{j}{n}\mathbb{P}(S_n=j).
	\end{equation*}
\end{lem}

Our proofs of the lower bounds in Theorem \ref{mainthm} will be more complicated than those for the upper bounds, although they still use only robust probabilistic techniques such as a generalised ballot theorem, Poisson approximations for the binomial distribution, and Brownian approximations to random walks. In future work we intend to demonstrate the adaptability of our new approach by applying our methods to other random graph models. As a first step in this direction, for applications of Lemma \ref{ballotlemma} to a random intersection graph, an inhomogeneous random graph, and percolation on a $d$-regular graph, see \cite{de_ambroggio:component_sizes_crit_RGs}.

\vspace{3mm}

\textbf{Structure of the paper}. We start by introducing ballot-type results in Section \ref{ballot_sec}, where we prove Lemma \ref{ballotlemma} and a corollary which will be the main tool to obtain the upper bounds in Theorem \ref{mainthm}. We also state a generalised ballot theorem due to Addario-Berry and Reed \cite{addario_berry_reed:ballot_theorems} that will be used for our lower bounds. Subsequently, in Section \ref{UBsec} we prove the upper bounds in $(a)$ and $(b)$ of Theorem \ref{mainthm}, whereas the corresponding lower bounds will be proved in Section \ref{LBsec}.

\vspace{3mm}

\textbf{Notation}. We write $\N_0 = \N\cup\{0\}$, $[n]=\{1,2,\ldots,n\}$, and $\llbracket a,b\rrbracket = [a,b]\cap\mathbb Z$. The abbreviation i.i.d.~means ``independent and identically distributed''. The empty sum is defined to be $0$, and the empty product is defined to be $1$. In particular we use the convention that $\sum_{i=n+1}^n a_i$ is zero, for any $n$ and any sequence $(a_i)$. For brevity we simply write $A$ rather than $A(n)$, $\lambda$ instead of $\lambda(n)$, and $p$ in place of $p(n)=1/n + \lambda n^{-4/3}$. We will often write $c$ to mean a constant in $(0,\infty)$, and use $c$ many times in a single proof even though the constant may change from line to line.

\subsection{Related work}
Besides his Proposition 2, which gives asymptotics for $\P(|\mathcal C_{\max}|< an^{2/3})$ in the case of constant (large) $A$ and $\lambda$, Pittel \cite{pittel:largest_cpt_rg} includes several other results which we make no attempt to rework. These include asymptotics for $\P(|\mathcal C_{\max}|< an^{2/3})$ when $a$ is small. Nachmias and Peres \cite{nachmias_peres:CRG_mgs} also gave a simple but inaccurate upper bound on this quantity, and it would be interesting to give an intuitive probabilistic proof of more accurate asymptotics. Pittel's paper is partially based on an earlier article by Luczak, Pittel and Wierman \cite{luczak_et_al:structure_RG}.

For $G(n,p)$ outside the critical scaling window, i.e.~when $\lambda$ is not bounded in $n$, $n^{2/3}$ is not the most likely size for the largest component of the graph, and therefore our results---while still true, at least provided $|\lambda| \le A/3 = o(n^{1/30})$---appear less natural than those by Nachmias and Peres \cite{nachmias_peres:outside_scaling_window}, Bollob\'as and Riordan \cite{bollobas_riordan:asymptotic_normality_RG} or Riordan \cite{riordan:phase_transition_config}.

A local limit theorem for the size of the $k$ largest components (for arbitrary $k$) was given by Van der Hofstad, Kager and M\"uller \cite{hofstad_et_al:local_limit_CRG}. See also Van der Hofstad, Kleim and Van Leeuwaarden \cite{hofstad_et_al:cluster_tails_power_law_RGs}, where similar results to those established by Pittel \cite{pittel:largest_cpt_rg} are proved in the context of inhomogeneous random graphs.

Aldous \cite{aldous:critical_random_graphs} used a \textit{breadth-first} search algorithm to explore $G(n,p)$ for $p$ within the critical window, and showed that the sizes of the largest components, if rescaled by $n^{2/3}$, converge (in an appropriate sense) to some limit, which he described in detail. The same type of argument has been used by Van der Hofstad \cite{hofstad_et_al:critical_epidemics} to investigate critical SIR epidemics. The work of Aldous was then developed by Addario-Berry, Broutin and Goldschmidt \cite{broutin_et_al:continuum_limit_critical_rgs} who showed that the rescaled components themselves converge to metric spaces characterised by excursions of Brownian motion with parabolic drift, decorated by a Poisson point process.

There are several other models that share similar properties with the near-critical Erd\H{o}s-R\'enyi graph. For instance, there are many critical models whose component sizes, when suitably rescaled, converge to the lengths of excursions of Brownian motion with parabolic drift just as for the Erd\H{o}s-R\'enyi graph. Some examples include inhomogeneous random graphs (see e.g. \cite{bhamidi_et_al:scaling_inhom_RGs} and \cite{bhamidi_et_al:scaling_limits_crit_inhom_RG}), the configuration model (see \cite{dhara_et_al:critical_window_config}, \cite{joseph:components_critical_RG} and \cite{riordan:phase_transition_config}), and quantum random graphs (see \cite{dembo_et_al:component_sizes_quantum_RG}).

In another direction we mention \cite{oconnell:LDs_RGs},where a large deviations rate function is provided for the size of the maximal component divided by $n$, valid for the $G(n,\gamma/n)$ model with $\gamma>0$. For a very recent work in this direction, see \cite{andreis_konig_patterson:large_devs_ER}.

Finally, the results of \cite{roberts:component_ER} were used to show the existence of times when a dynamical version of the Erd\H{o}s-R\'enyi graph has an unusually large connected component. Related results about the structure of dynamical Erd\H{o}s-R\'enyi graphs were given by Rossignol \cite{rossignol:scaling_limit_dynamical_ER}.

\section{Ballot-style results}\label{ballot_sec}
Let $X_1,\dots,X_n \in \{-1,1\}$ be i.i.d.~random variables taking values in $\{-1,1\}$, with $\mathbb{P}(X_i=1)=\mathbb{P}(X_i=-1)=1/2$, and let $S_t = \sum_{i=1}^{t}X_i$. In its simplest form the ballot theorem concerns the probability that $S_t$ stays positive for all times $t\in [n]$, given that $S_n=k \in \mathbb{N}$, and says that the answer is $k/n$;  see e.g. \cite{addario_berry_reed:ballot_theorems, kager:hitting_time, konstantopoulos:ballot, remco:hitting_time} and references therein. However, we will be interested in evaluating probabilities of the following type:
\begin{equation*}
\mathbb{P}\left(1+S_t>0 \hspace{0.2cm}\forall t\in [n] ,\, 1+S_n=k\right),
\end{equation*}
where $k\geq 1$ and $X_1,\dots,X_n$ are i.i.d. random variables taking values in $\{-1,0,1,2,\dots\}$. A possible solution might be to apply the following generalised ballot theorem.
\begin{thm}[Addario-Berry and Reed \cite{addario_berry_reed:ballot_theorems}]\label{genballot}
	Suppose $X$ is a random variable satisfying $\mathbb{E}[X]=0$, $\text{Var}(X)>0$, $\mathbb{E}[X^{2+\alpha}]<\infty$ for some $\alpha >0$, and $X$ is a lattice random variable with period $d$ (meaning that $dX$ is an integer random variable and $d$ is the smallest positive real number for which this holds). Then given independent random variables $X_1,X_2,\dots$ distributed as $X$ with associated partial sums $S_t = \sum_{i=1}^{t}X_i$, for all $j$ such that $0\leq j =O\left(\sqrt{n}\right)$ and such that $j$ is a multiple of $1/d$ we have 
	\begin{equation*}
	\mathbb{P}\left(S_t > 0 \hspace{0.15cm}\forall t\in [n],S_n=j\right)=\Theta\left(\frac{j+1}{n^{3/2}}\right).
	\end{equation*}
\end{thm}
This result will indeed be useful in the proof of the lower bounds in our Theorem \ref{mainthm}. However, for the upper bound we will need a result that holds when $j$ is much larger than $\sqrt{n}$. Our Lemma \ref{ballotlemma} shows that the upper bound remains true more generally. We now aim to prove that result.

Fix $n\in \mathbb{N}$. Let $X = (X_1,\dots,X_n)$ be random variables taking values in $\mathbb{Z}$. Define $S_0 = 0$ and $S_t = \sum_{i=1}^{t}X_i$ for all $t\in [n]$. Given $r\in [n]$, define the \textit{rotation} of $S=(S_0,S_1,\dots,S_n)$ by $r$ as the walk $S^{r}=(S_0^{r},S_1^{r},\dots,S_n^{r})$ corresponding to the rotated sequence $X^{r} = (X_{r+1},\dots,X_n,X_1,\dots,X_r)$. That is,
\begin{itemize}
	\item if $0\leq t\leq n-r$, then $S_t^{r} = S_{t+r}-S_r = \sum_{i=r+1}^{t+r}X_i$;
	\item if $n-r<t\leq n$, then $S_t^{r} = S_n+S_{t+r-n}-S_r = \sum_{i=r+1}^{n}X_i+\sum_{i=1}^{t+r-n}X_i$.
\end{itemize}
In particular, $S_n^{r}=\sum_{i=r+1}^{n}X_i+\sum_{i=1}^{r}X_i=S_n$ (for every $r\in [n]$) and $S^n=S$.
\begin{defn}
	We say that $r\in [n]$ is \textit{favourable} if $S_t^{r}>0 $ for every $t\in [n]$.
\end{defn}
The following lemma contains the key observation needed to prove Lemma \ref{ballotlemma}.
\begin{lem}\label{favlem}
	Fix $j\in \mathbb{N}$. If $S_n=j$, then
	\[|\{r\in[n]: r \text{ is favourable}\}|\leq j.\]
\end{lem}
\begin{proof}
	Let $1\leq I_1<\dots <I_L\leq n$ denote the indices (if any) such that $I_k$ is favourable for $1\leq k\leq L$. We need to show that $L\leq j$. Observe that $S_{I_{k+1}-I_k}^{I_k}\geq 1$ for $1\leq k\leq L-1$. Therefore we get $\sum_{k=1}^{L-1}S_{I_{k+1}-I_k}^{I_k}\geq L-1$. By the same argument, $S_{(I_1+n)-I_L}^{I_L}\geq 1$. Consequently
	\begin{equation*}
	L=(L-1)+1\leq  \sum_{k=1}^{L-1}S_{I_{k+1}-I_k}^{I_k}+S_{(I_1+n)-I_L}^{I_L}=S_n=j.
	\qedhere
	\end{equation*}
\end{proof}
\begin{proof}[Proof of Lemma \ref{ballotlemma}]
For any $r\in[n]$, since $(X_1,X_2,\ldots,X_n)$ is invariant under rotations, 
\[\mathbb{P}(S_t>0\hspace{0.15cm}\forall t\in [n],\,S_n=j) = \mathbb{P}(S_t^r>0\hspace{0.15cm}\forall t\in [n],\, S_n^r=j) = \mathbb{P}(r \text{ is favourable},\, S_n^r=j)\]
and since $S_n^r=S_n$, we obtain that
\[\mathbb{P}(S_t>0\hspace{0.15cm}\forall t\in [n],\,S_n=j)=\mathbb{P}(r\text{ is favourable},\,S_n=j).\]
	Summing over $r\in[n]$ and applying Lemma \ref{favlem} we have
	\begin{align*}
	n\mathbb{P}(S_t>0\hspace{0.15cm}\forall t\in [n],\,S_n=j) & = \sum_{r=1}^{n}\mathbb{E}[\ind_{\{r \text{ is favourable}\}}\ind_{\{S_n=j\}}] \\
	& = \mathbb{E}\bigg[\ind_{\{S_n=j\}}\sum_{r=1}^{n}\ind_{\{r \text{ is favourable}\}}\bigg]\\
	&\leq \mathbb{E}\left[\ind_{\{S_n=j\}}j\right]= j\mathbb{P}(S_n=j)
	\end{align*}
	which completes the proof.
\end{proof}

The following corollary will be used to prove the upper bounds of Theorem \ref{mainthm}.
\begin{cor}\label{ballotcor}
		Fix $n\in\N$ and let $(X_i)_{i\ge 1}$ be i.i.d.~random variables taking values in $\mathbb{Z}$, whose distribution may depend on $n$. Let $h\in \mathbb{N}$, and suppose that $\mathbb{P}(X_1=h)>0$. Define $S_t = \sum_{i=1}^{t}X_i$ for $t\in\mathbb{N}_0$. Then for any $j\geq 1$ we have
		\begin{equation*}
		\mathbb{P}(h+S_t>0\hspace{0.15cm} \forall t\in [n],\, h+S_{n}=j)\leq \mathbb{P}(X_1=h)^{-1}\frac{j}{n+1}\mathbb{P}(S_{n+1}=j).
		\end{equation*}
	\end{cor}
	\begin{proof}
		Let $X_0$ be an independent copy of $X_1$. Define $S_t^*=X_0+S_t$ for $0\leq t\leq n$. Then
		\begin{align}
		&\mathbb{P}(h+S_t>0\hspace{0.15cm} \forall t\in [n],\, h+S_{n}=j)\nonumber\\
		&\hspace{40mm}=\mathbb{P}(X_0=h)^{-1}\mathbb{P}(h+S_t>0\hspace{0.15cm} \forall t\in [n],\,h+S_{n}=j,\,X_0=h)\nonumber\\
		&\hspace{40mm}=\mathbb{P}(X_1=h)^{-1}\mathbb{P}(S_t^*>0\hspace{0.15cm} \forall t\in [n],\,S_{n}^*=j,\,S_0^*=h)\nonumber\\
		&\hspace{40mm}\leq \mathbb{P}(X_1=h)^{-1}\mathbb{P}(S_t^*>0\hspace{0.15cm} \forall t\in \{0\}\cup [n],\,S_{n}^*=j).\label{ballotcoreq}
		\end{align}
		Now since $(S_0^*,S_1^*,\dots,S_n^*)\overset{d}{=}(S_1,S_2,\dots,S_{n+1})$, applying Lemma \ref{ballotlemma} we obtain that
		\begin{align*}
		\mathbb{P}(S_t^*>0\hspace{0.15cm} \forall t\in \{0\}\cup [n],\,S_{n}^*=j)&=\mathbb{P}(S_t>0\hspace{0.15cm} \forall t\in [n+1]\,,S_{n+1}=j)\\
		&\leq \frac{j}{n+1}\mathbb{P}(S_{n+1}=j),
		\end{align*}
		and substituting this into \eqref{ballotcoreq} gives the result.
\end{proof}

\section{Proof of the upper bounds in Theorem \ref{mainthm}}\label{UBsec}

A main ingredient in our analysis is an \textit{exploration process}, which is a procedure to sequentially discover the component containing a given vertex, and which reduces the study of component sizes to the analysis of the trajectory of a stochastic process. Such exploration processes are well-known, dating back at least to \cite{martin-lof:symmetric_sampling}, and several variants exist. Our description closely follows the one appearing in \cite{roberts:component_ER}; see also \cite{nachmias_peres:CRG_mgs}.

Let $G$ be any (undirected) graph with vertex set $[n]$, and let $v\in [n]$ be any given vertex. Fix an ordering of the $n$ vertices with $v$ first. At each time $t\in \{0\}\cup [n]$ of the exploration, each vertex will be \textit{active}, \textit{explored} or \textit{unseen}; the number of explored vertices will be $t$ whereas the (possibly random) number of active vertices will be denoted by $Y_t$. At time $t=0$, vertex $v$ is declared to be active whereas all other vertices are declared unseen, so that $Y_0=1$. At each step $t\in [n]$ of the procedure, if $Y_{t-1}>0$ then we let $u_t$ be the first active vertex; if $Y_{t-1}=0$, we let $u_t$ be the first unseen vertex (here the term \textit{first} refers to the ordering that we fixed at the beginning of the procedure). Note that at time $t=1$ we have $u_1=v$. Denote by $\eta_t$ the number of unseen neighbours of $u_t$ in $G$ and change the status of these vertices to active. Then, set $u_t$ itself explored. From this description we see that:
\begin{itemize}
	\item $Y_t=Y_{t-1}+\eta_t-1$, if $Y_{t-1}>0$;
	\item $Y_t=\eta_t$, if $Y_{t-1}=0$.
\end{itemize}
We now specialize to the Erd\H{o}s-R\'enyi random graph, i.e.~we now take $G=G(n,p)$. Let us denote by $U_t=n-Y_t-t$ the number of unseen vertices in $G(n,p)$ at time $t$, and define $\mathcal{F}_0 = \{\Omega,\emptyset\}$ and $\mathcal{F}_t = \sigma(\{\eta_j:1\leq j\leq t\})$ for $t\in [n]$. Then for $t\in [n]$, given $\mathcal{F}_{t-1}$, we see that $\eta_t\sim\Bin(U_{t-1},p)$. 
Since $U_t\leq n-t$, we can couple the process $(\eta_i)_{i \in [n]}$ with a sequence $(\tau_i)_{i \in [n]}$ of independent $\Bin(n-i,p)$ random variables such that $\tau_i\geq \eta_i$ for all $i$. It follows that, for any $k\in[n]$,
	\begin{align}
	\mathbb{P}(|\mathcal{C}(v)|> k) &= \mathbb{P}(Y_t>0\hspace{0.15cm}\forall t\in [k])\nonumber\\
	&=\mathbb{P}\bigg(1+\sum_{i=1}^{t}(\eta_i-1)>0\hspace{0.15cm}\forall t\in [k]\bigg)\nonumber\\
	&\leq \mathbb{P}\bigg(1+\sum_{i=1}^{t}(\tau_i-1)>0\hspace{0.15cm}\forall t\in [k] \bigg).\label{CtoTauEq}
	\end{align}
	We would like to apply Corollary \ref{ballotcor} to the sequence $(1+\sum_{i=1}^{t}(\tau_i-1))_{t\in [k]}$, and to this end we need to turn the latter process into a random walk with identically distributed increments. This is achieved in Lemma \ref{tautoR} below.
	\begin{lem}\label{tautoR}
		There exists a finite constant $c$ such that for any $k\in[n]$,
		\[\mathbb{P}\bigg(1+\sum_{i=1}^{t}(\tau_i-1)>0\hspace{0.15cm}\forall t\in [k] \bigg) \leq c\hspace{0.2mm}\mathbb{P}\left(1+R_t>0\hspace{0.15cm}\forall t\in [k],\, 1+R_{k}\ge \frac{k^2p}{2} - \frac{k}{n^{1/2}}\right),\]
		where $(R_t)_{t\geq 0}$ is a random walk with $R_0=0$ and i.i.d.~steps each having distribution $\Bin(n,p)-1$.
	\end{lem}	
		The idea behind this lemma is that by adding an independent $\Bin(i,p)$ random variable to $\tau_i$, we transform it into a $\Bin(n,p)$ random variable which forms one of the steps of the random walk $R_t$ appearing on the right-hand side. If the sum of the $\tau_i$ up to $t$ remains positive then $R_t$, which is larger, must certainly also remain positive; but also the final value $R_k$ must be larger than the sum of the additional contributions from the $\Bin(i,p)$ random variables. A standard bound shows that these additional contributions are concentrated about their mean, which is approximately $k^2p/2$.

	We postpone the details until Section \ref{lem3sec}, and continue with the proof of the upper bounds in Theorem \ref{mainthm}. By summing over the possible values of $R_k$, we can apply Corollary \ref{ballotcor} with $h=1$ to the quantity on the right-hand side of Lemma \ref{tautoR}: it is at most
\begin{equation}\label{applyBallotCor}
\frac{c}{k+1}\sum_{j=h(k,n)}^{(k+1)(n-1)} j \mathbb{P}\left( R_{k+1}=j\right),
\end{equation}
where $h(k,n)=\lceil \frac{k^2}{2}p - \frac{k}{n^{1/2}} \rceil$, and the upper limit on the sum is due to the fact that $R_{k+1}\leq (k+1)(n-1)$ (because each step of $R_t$ is at most $n-1$, and in $R_{k+1}$ we are summing $k+1$ of them).

We now rewrite the above sum in a way that is easier to analyse, using the following elementary observation. If $X$ is a random variable taking values in $\mathbb{Z}\cap(-\infty,N]$ for some $N\in \N$, then for any $h\ge1$, we have
\begin{align*}
	\E[X\ind_{\{X\ge h\}}] = \E\Big[\sum_{i=1}^N \ind_{\{i\le X\}}\ind_{\{X\ge h\}}\Big]&= \E\Big[\sum_{i=1}^h \ind_{\{X\ge h\}}\Big] + \E\Big[\sum_{i=h+1}^N \ind_{\{X\ge i\}}\Big]\\
	&= h\P(X\ge h) + \sum_{i=h+1}^N \P(X\ge i).
	\end{align*}

Applying this to $R_{k+1}$, and using that $h(k,n)/(k+1) \le k/n$ when $n$ is large, we have
	\[\frac{1}{k+1}\sum_{j=h(k,n)}^{(k+1)(n-1)} j \mathbb{P}\left( R_{k+1}=j\right) \le \frac{k}{n} \P(R_{k+1}\ge h(k,n)) + \frac{1}{k+1} \sum_{j=h(k,n)+1}^{(k+1)(n-1)} \P(R_{k+1}\ge j),\]
and putting this together with \eqref{CtoTauEq}, Lemma \ref{tautoR} and \eqref{applyBallotCor}, we have shown that
	\begin{equation*}
	\mathbb{P}\left(|C(v)|>k \right)\leq \frac{ck}{n} \P(R_{k+1}\ge h(k,n)) + \frac{c}{k+1} \sum_{j=h(k,n)+1}^{(k+1)(n-1)} \P(R_{k+1}\ge j).
	\end{equation*}
	The next two lemmas conclude the proof of the upper bound in part $(a)$ of Theorem \ref{mainthm} by showing that, when we take $k=\lceil An^{2/3}\rceil$ with $A\ge1$, the right-hand side above is bounded by $cA^{-1/2}n^{-1/3}\exp\{-A^3/8+\lambda A^2/2-\lambda^2A/2\}$. Let
	\[H(A,n) = h(\lceil An^{2/3}\rceil,n) = \Big\lceil \frac{\lceil An^{2/3}\rceil^2}{2}p - \frac{\lceil An^{2/3}\rceil}{n^{1/2}} \Big\rceil.\]
	
	\begin{lem}\label{uppertechlem1}
	Suppose that $1\le A=o\left(n^{1/12}\right)$, $\lambda = o(n^{1/12})$ and $\lambda\le A/3$. There exists a finite constant $c$ such that
	\begin{equation*}
	\frac{\lceil An^{2/3}\rceil}{n} \P\Big(R_{\lceil An^{2/3}\rceil+1}\ge H(A,n)\Big)\\
	\le \frac{c}{A^{1/2} n^{1/3}}e^{-A^3/8+\lambda A^2/2 - \lambda^2 A/2}.
	\end{equation*}
	\end{lem}
	
	\begin{lem}\label{uppertechlem2}
	Suppose that $1\le A=o\left(n^{1/12}\right)$, $\lambda = o(n^{1/12})$ and $\lambda\le A/3$. There exists a finite constant $c$ such that
	\[\frac{1}{\lceil An^{2/3}\rceil+1} \sum_{j=H(A,n)+1}^{(\lceil An^{2/3}\rceil+1)(n-1)} \P(R_{\lceil An^{2/3}\rceil+1}\ge j) \leq \frac{c}{A^2 n^{1/3}}e^{-A^3/8+\lambda A^2/2 - \lambda^2 A/2}.\]
	\end{lem}
	
	Since $R_{\lceil An^{2/3}\rceil+1}$ is simply a binomial random variable, the proofs of Lemmas \ref{uppertechlem1} and \ref{uppertechlem2} are exercises in applying standard estimates to binomial random variables. We carry out the details in Section \ref{lem3sec}. Subject to these and the proof of Lemma \ref{tautoR}, the proof of the upper bound in part (a) of Theorem \ref{mainthm} is complete.
	
	The upper bound of part (b) in Theorem \ref{mainthm} is deduced from the upper bound in part (a) using the following standard procedure, used for example in \cite{nachmias_peres:CRG_mgs}. For any $k\in[n]$, denote by
	\begin{equation*}
	N_{k}=\sum_{i=1}^{n}\ind_{\{|\mathcal{C}(v_i)|> k \}}
	\end{equation*} 
	the number of vertices that are contained in components of size larger than $k$. If $u$ is any fixed vertex in $G(n,p)$, we have
	\[\mathbb{P}(|\mathcal{C}_{\max}|> k) = \mathbb{P}(N_{k}> k)\leq \frac{1}{k}\mathbb{E}[N_{k}] = \frac{n}{k}\mathbb{P}(|\mathcal{C}(u)|>k)\]
	and then taking $k=\lceil An^{2/3}\rceil$ and applying part $(a)$, this is at most
	\[\frac{n}{\lceil An^{2/3}\rceil}\frac{c}{A^{1/2}n^{1/3}}e^{-A^3/8+\lambda A^2/2 - \lambda^2 A/2}\leq cA^{-3/2}e^{-A^3/8+\lambda A^2/2 - \lambda^2 A/2},\]
	as required. This concludes the proof for the upper bounds (a) and (b) in Theorem \ref{mainthm}, subject to proving Lemmas \ref{tautoR}, \ref{uppertechlem1} and \ref{uppertechlem2}.

	\subsection{Proofs of Lemmas \ref{tautoR}, \ref{uppertechlem1} and \ref{uppertechlem2}}\label{lem3sec}
	
	To prove Lemmas \ref{tautoR}, \ref{uppertechlem1} and \ref{uppertechlem2} we will make use of the following two preliminary results on the concentration of Binomial random variables about their mean. The first of these results is Theorem 1.6(ii) in \cite{bollobas_book}, while the second is Theorem 2.1 in \cite{janson_et_al:random_graphs}.
	\begin{lem}\label{Bol1}
		Let $S\sim \Bin(n,p)$ and suppose that $0<p= p(n)<1$ satisfies $np(1-p)\rightarrow \infty$ as $n\rightarrow \infty$. If $x=x(n)\to\infty$ but $x(n)=o((np(1-p))^{1/6})$, then
		\begin{equation*}
		x e^{x^2/2}\mathbb{P}(S\geq np + x(np(1-p))^{1/2}) \to (2\pi)^{-1/2}
		\end{equation*}
		as $n\to\infty$.
	\end{lem}
	
	\begin{lem}\label{Bol2}
		Let $S\sim Bin(n,p)$ and define $\phi(x) = (1+x)\log(1+x)-x$ for $x\geq -1$. Then for every $t\geq 0$ we have that
		\begin{itemize}
			\item [(a)] $\mathbb{P}(S\geq \mathbb{E}[S]+t)\leq \exp\{-\mathbb{E}[S]\phi(t/\mathbb{E}[S])\}\leq \exp\left\{-t^2/2(\mathbb{E}[S]+t/3)\right\}$;
			\item [(b)] $\mathbb{P}(S\leq \mathbb{E}[S]-t)\leq \exp\{-t^2/2\mathbb{E}[S]\}$.
		\end{itemize}
	\end{lem}
	
	We are now ready to start with the proofs of the lemmas stated in the previous section.
	
		\begin{proof}[Proof of Lemma \ref{tautoR}]
		We want to bound
		\begin{equation*}
		\mathbb{P}\Big(1+\sum_{i=1}^{t}(\tau_i-1)>0\hspace{0.15cm}\forall t\in [k] \Big)
		\end{equation*}
		from above, where $\tau_i\sim \Bin(n-i,p)$ are independent. We do this by adding extra terms to the sum $\sum_{i=1}^{t}(\tau_i-1)$ to create a random walk with identically distributed steps. To this end, let $(B_i)_{i\in [n]}$ be a sequence of independent random variables, also independent from $(\tau_i)_{i \in [n]}$, and such that $B_i\sim \Bin(i,p)$ for every $i \in [n]$. Moreover, define $S_t = \sum_{i=1}^{t}B_i$ for $t\in [n]$. Let
		\begin{equation}
		P = \mathbb{P}\Big(S_{k}\ge \frac{k^2}{2}p-\frac{k}{n^{1/2}}\Big).
		\end{equation}
		Since $S_{k} \sim \Bin\left(k(k+1)/2,p\right)$, an application of Lemma \ref{Bol2}(b) with $t=kn^{-1/2} + kp/2$ yields that $P\geq c$ for some $c>0$. Now using the independence of $(\tau_i)_{i\in [n]}$ and $(B_i)_{i \in [n]}$ we obtain that 
		\begin{equation*}
		\mathbb{P}\bigg(1+\sum_{i=1}^{t}(\tau_i-1)>0\hspace{0.15cm}\forall t \in [k] \bigg) =P^{-1}\mathbb{P}\bigg(1+\sum_{i=1}^{t}(\tau_i-1)>0\hspace{0.15cm}\forall t \in [k],\, S_{k}\ge\frac{k^2}{2}p-\frac{k}{n^{1/2}}\bigg).
		\end{equation*}
		Setting $R_t = \sum_{i=1}^{t}(\tau_i + B_i-1)$, we see that the last quantity is bounded from above by 
		\begin{align*}
		c^{-1}\mathbb{P}\left(1+R_t>0\hspace{0.15cm}\forall t \in [k], 1+R_{k}\ge\frac{k^2}{2}p - \frac{k}{n^{1/2}}\right)
		\end{align*}
	so noting that $\tau_i+B_i \overset{iid}{\sim}\Bin(n,p)$ for every $i\in [n]$ completes the proof.
	\end{proof}

\begin{proof}[Proof of Lemma \ref{uppertechlem1}]
	Write
	\[K=K(A,n) = \lceil An^{2/3}\rceil+1\]
	and recall that
	\[H=H(A,n) = \Big\lceil \frac{\lceil An^{2/3}\rceil^2}{2}p - \frac{\lceil An^{2/3}\rceil}{n^{1/2}} \Big\rceil.\]
	We want to use Lemma \ref{Bol1} to bound from above the quantity
	\begin{equation*}
	\frac{K-1}{n} \P(R_{K}\ge H).
	\end{equation*}
	The first step is to rewrite the above probability so that is in the form appearing in Lemma \ref{Bol1}. Letting $B_{j,p}$ be a binomial random variable with parameters $j$ and $p$, we have
	\begin{align*}
	\P(R_{K}\ge H) &= \mathbb{P}( B_{nK,p} \geq K+H)\\
	&=\mathbb{P}( B_{nK,p} \geq nKp+H-K\lambda/n^{1/3})\\
	&= \P\big(B_{nK,p} \geq nKp + x(A,n,\lambda)\sqrt{nKp(1-p)}\big)
	\end{align*}
	where we define
	\[x(A,n, \lambda) = \frac{H-K\lambda/n^{1/3}}{\sqrt{nKp(1-p)}}.\]
	Elementary estimates using the fact that $A=o(n^{1/12})$ and $\lambda=o(n^{1/12})$ show that
	\[x(A,n, \lambda) = \frac{A^{3/2}}{2} - \lambda A^{1/2} + o(A^{1/2}n^{-1/6}).\]
	Applying Lemma \ref{Bol1} and using the fact that $\lambda \le A/3$, we obtain that, for large $n$,
	\begin{align*}
	\frac{K-1}{n} \P(R_{K}\ge H) &\le \frac{cA}{n^{1/3}} \frac{1}{x(A,n,\lambda)}e^{-x(A,n\lambda)^2/2}\\
	&\le \frac{c}{\sqrt{A}n^{1/3}}\exp\left(-\frac{A^3}{8}+\frac{\lambda A^2}{2} - \frac{\lambda^2 A}{2} \right),
	\end{align*}
	which completes the proof of Lemma \ref{uppertechlem1}.
	\end{proof}
	Before we prove Lemma \ref{uppertechlem2}, we will need the following bound, which is an easy application of Lemma \ref{Bol2}.
		\begin{lem}\label{Rkger}
		Suppose that $B_{N,p}$ is a binomial random variable with parameters $N\ge 1$ and $p\in[0,1]$. Let $C\in(0,\infty)$ be constant. Then for all $x\in(0, C(Np)^{2/3}]$ we have that 
		\[\P\left(B_{N,p}\ge Np + x\right)\leq c\exp\left(-\frac{x^2}{2Np}\right)\]
		where $c$ is another finite constant.
	\end{lem}
	
	\begin{proof}
		Applying Lemma \ref{Bol2}, we have
		\[\P\left(B_{N,p}\ge Np + x\right) \le \exp\Big(-Np\Big[\Big(1+\frac{x}{Np}\Big)\log\Big(1+\frac{x}{Np}\Big) - \frac{x}{Np}\Big]\Big),\]
		and since $\log(1+t)>t-t^2/2$ for every $t>0$,
		\begin{align*}
		\P\left(B_{N,p}\ge Np + x\right) &\le \exp\Big(-Np\Big[\Big(1+\frac{x}{Np}\Big)\Big(\frac{x}{Np}-\frac{x^2}{2(Np)^2}\Big)-\frac{x}{Np}\Big]\Big)\\
		&= \exp\Big(-Np\Big[\frac{x^2}{2(Np)^2} - \frac{x^3}{2(Np)^3}\Big]\Big)\\
		&= \exp\Big(-\frac{x^2}{2Np}+\frac{x^3}{2(Np)^2}\Big),
		\end{align*}
		which establishes the result with $c=\exp(C^3/2)$.
	\end{proof}
	
	\begin{proof}[Proof of Lemma \ref{uppertechlem2}]
	Writing
	\[K = K(A,n) = \lceil An^{2/3}\rceil+1 \hspace{4mm} \text{and} \hspace{4mm} H = H(A,n) = \Big\lceil \frac{\lceil An^{2/3}\rceil^2}{2n} - \frac{\lceil An^{2/3}\rceil}{n^{1/2}} \Big\rceil,\]
	we aim to bound
	\begin{equation*}
	\frac{1}{K}\sum_{j=H+1}^{K(n-1)}\mathbb{P}(R_{K} \ge j)
	\end{equation*}
	from above. We first note that
	\begin{equation}\label{UTL2breakdown}
	\frac{1}{K}\sum_{j=H+1}^{K(n-1)}\mathbb{P}(R_{K} \ge j) \le \frac{1}{K}\sum_{j=H+1}^{\lfloor K^{2/3}\rfloor} \P(R_K\ge j) + n\P(R_K \ge K^{2/3}).
	\end{equation}
	To bound the second term on the right-hand side of \eqref{UTL2breakdown} observe that, since $A=o(n^{1/12})$ and $\lambda = o(n^{1/12})$, we have $K\ge nKp - K^{2/3}/2$ when $n$ is large. Thus, when $n$ is large,
	\begin{align}\label{vsmall}
	n\P(R_K \ge K^{2/3}) &= n\P(B_{nK,p}\ge K + K^{2/3})\nonumber\\
	&\leq n\P(B_{nK,p}\ge nKp+K^{2/3}/2).
	\end{align}
	Using the second inequality in part (a) of Lemma \ref{Bol2} we obtain
	\begin{align}\label{dd}
	(\ref{vsmall})&\le n \exp\left\{-\frac{K^{4/3}}{8(nKp+\frac{1}{6}K^{2/3})}\right\} \leq n\exp\left\{-cA^{1/3}n^{2/9}\right\}
	\end{align}
	and for sufficiently large $n$ we have that
	\begin{align*}
	(\ref{dd})\leq \frac{1}{A^2n^{1/3}}e^{-\frac{A^3}{8}+\frac{\lambda A^2}{2} - \frac{\lambda^2 A}{2}}.
	\end{align*}
	Next, for the first term on the right-hand side of \eqref{UTL2breakdown}, note that
	\begin{align}\label{remaining}
	\frac{1}{K}\sum_{j=H+1}^{\lfloor K^{2/3}\rfloor} \P(R_K\ge j)&=\frac{1}{K}\sum_{j=H+1}^{\lfloor K^{2/3}\rfloor} \P(B_{nK,p}\ge K+j)\nonumber\\
	&=\frac{1}{K}\sum_{j=H+1}^{\lfloor K^{2/3}\rfloor} \P(B_{nK,p}\ge nKp+j+K\lambda n^{-1/3}).
	\end{align}
	Since $A=o(n^{1/12})$ and $\lambda=o(n^{1/12})$, we have $K\lambda n^{-1/3}=o(K^{2/3})$, and therefore we may apply Lemma \ref{Rkger} to obtain
	\begin{align*}
	(\ref{remaining})\le \frac{c}{K}\sum_{j=H+1}^{\lfloor K^{2/3} \rfloor}\exp\left(-\frac{(j+K\lambda n^{-1/3})^2}{2nKp}\right)\leq \frac{c}{\sqrt{K}}\P\left(G\ge \frac{H+1-K\lambda/n^{1/3}}{\sqrt{nKp}}\right),
	\end{align*}
	where $G$ denotes a Gaussian random variable with mean zero and unit variance. Recalling the standard bound $\P(G\geq t)\leq \left(t\sqrt{2\pi}\right)^{-1}e^{-t^2/2}$, which is valid for every $t>0$, we obtain
	\[\P\left(G\ge \frac{H+1-K\lambda/n^{1/3}}{\sqrt{nKp}}\right)	\leq \frac{1}{\sqrt{2\pi}}\frac{\sqrt{nKp}}{H+1-K\lambda/n^{1/3}}\exp\Big(-\frac{(H+1-K\lambda/n^{1/3})^2}{2nKp}\Big).\]
	An easy computation reveals that
	\[\frac{(H+1-K\lambda/n^{1/3})^2}{2nKp}\ge \frac{A^3}{8}-\lambda\frac{A^2}{2}+\lambda^2\frac{A}{2}+o(1),\]
	and consequently we obtain
	\begin{align*}
	\frac{c}{\sqrt{K}}\P\left(G\ge \frac{H+1-K\lambda/n^{1/3}}{\sqrt{nKp}}\right) \le \frac{c}{A^2n^{1/3}}\exp\left\{-\frac{A^3}{8}+\lambda\frac{A^2}{2}-\lambda^2\frac{A}{2}\right\},
	\end{align*}
	as required.
\end{proof}

\section{Proof of the lower bounds in Theorem \ref{mainthm}}\label{LBsec}

Let $v\in [n]$ be any vertex in $G(n,p)$ from which we start running the exploration process described at the beginning of Section \ref{UBsec}. We write $T_2 = \lceil An^{2/3}\rceil$; we will in due course also have a time $T_1$ which is smaller than $T_2$.

Recall from Section \ref{UBsec} that $\eta_{i}$ denotes the number of unseen vertices which become active during the $i$th step of the exploration process, and $Y_t=1+\sum_{i=1}^{t}(\eta_i-1)$ is the number of active vertices at step $t$ of the procedure. Moreover, recall that
\begin{align}\label{conddistr}
\eta_i|\mathcal{F}_{i-1}\sim \Bin(n-i+1-Y_{i-1},p),
\end{align} 
where $\mathcal{F}_i = \sigma(\{\eta_{1},\dots,\eta_i\})$ and $p=p(n)=1/n+\lambda n^{-4/3}$. We will start by proving the lower bound in part $(a)$ of Theorem \ref{mainthm}; that is, by bounding from below the probability
\begin{equation*}
\mathbb{P}\left(|C(v)|>T_2\right)=\mathbb{P}\left(1+\sum_{i=1}^{t}(\eta_i-1)>0 \hspace{0.2cm}\forall t \in [T_2]\right).
\end{equation*} 
We note that the random variables $\eta_i$ are not independent, which makes our analysis more difficult. The first part of our argument, therefore, consists of replacing the $\eta_{i}$ with a sequence of independent binomial random variables which are easier to analyse. The idea is that $\eta_i$, which is the number of neighbours of our $i$th vertex that are unseen, is roughly $\Bin(n-i,p)$, but the first parameter is slightly smaller due to the (random) number of active vertices that are present at the beginning of the $i$th step of the exploration process. If we can bound the number of active vertices above by some deterministic value $K$ with high probability, then we can remove this source of randomness and obtain a sequence of independent increments in place of the $\eta_i$. To this end, fix $K\in\N$ and suppose that $(\delta_i)_{i\in [T_2]}$ is a sequence of independent random variables with $\delta_i\sim \Bin(n-K-i,p)$, and set $R_t = 1+\sum_{i=1}^{t}(\delta_i-1)$. We note that the definitions of $\delta_i$ and $R_t$ depend implicitly on $K$; sometimes for clarity we will write $\delta_i^{(K)}$ and $R_t^{(K)}$. We will soon fix $K=\lfloor n^{2/5}\rfloor$, but the following lemma works for any $K<n-T_2$. We postpone the proof, which constructs a coupling between $\eta_i$ and $\delta_i$, until Section \ref{4142sec}.

\begin{lem}\label{first}
	Suppose that $K+T_2<n$. Then
	\begin{equation}\label{pp}
	\mathbb{P}\left(|C(v)|>T_2\right)\geq \mathbb{P}\left(R_t^{(K)}>0 \hspace{0.2cm}\forall t\in [T_2] \right)-\P\left(\exists i \in [T_2]:Y_i\geq K\right).
	\end{equation}
\end{lem}
Our next result shows that if we choose $K=\lfloor n^{2/5}\rfloor$, then we do not have to worry about the last probability on the right-hand side of (\ref{pp}).
\begin{lem}\label{rub1}
	As $n\rightarrow \infty$,
	\begin{equation}\label{azaz}
	\P\left(\exists i \in [T_2]:Y_i\geq  \lfloor n^{2/5}\rfloor \right)=o\left(A^{-1/2}n^{-1/3}e^{-\frac{A^3}{8}+\frac{\lambda A^2}{2}-\frac{\lambda^2A}{2}}\right).
	\end{equation} 
\end{lem}
The proof of Lemma \ref{rub1}, which easily follows from Lemma \ref{Bol2}, is again postponed to Section \ref{4142sec}.

Given (\ref{azaz}), we can now fix $K=\lfloor n^{2/5}\rfloor$ and focus on providing a lower bound for
\begin{equation}\label{a2}
\mathbb{P}\left(R_t^{(K)}>0 \hspace{0.15cm}\forall t\in [T_2] \right).
\end{equation}
Observe that, although we now have a process with independent increments, obtaining a lower bound for (\ref{a2}) remains a non-trivial task, because the $\delta_i$ that are used to define $R_t$ are not identically distributed. We consider two options to produce a random walk with i.i.d.~increments from $(R_t)_{t\in [T_2]}$. The first is to view $\delta_i$ as a sum of i.i.d.~Bernoulli random variables, with $\delta_1$ summing more Bernoullis than $\delta_2$ and so on; and then to rearrange the same Bernoullis amongst sums $\delta_i'$ that all have equal length. The second is simply to add an independent $\Bin(K+i,p)$ random variable to $\delta_i$ for each $i$.

It turns out that neither of these two options works on its own. The first has problems if we try to cover too many values of $i$, since the more Bernoullis that we have to rearrange, the less accurate our estimates become. The second has problems when $i$ is small, as the variance of the added $\Bin(K+i,p)$ random variables is too large when our random walk is near the origin.

We therefore combine the two techniques. We take $T_1\in[T_2]$, and carry out the first strategy for times $t\in [T_1]$, and the second strategy for $t\in \llbracket T_1, T_2\rrbracket$.

We note first that for any deterministic $H\in\N$ and $T_1\in[T_2]$,
\begin{align}
&\mathbb{P}\big(R_t>0 \hspace{0.15cm}\forall t \in [T_2]\big)\nonumber\\
&\hspace{5mm}\geq \mathbb{P}\big(R_t>0 \hspace{0.15cm}\forall t\in [T_1],\,\, R_{T_1}\in [H,2H],\,\, R_t>0\hspace{0.15cm}\forall t\in \llbracket T_1, T_2\rrbracket\big)\nonumber\\
&\hspace{5mm}\ge \mathbb{P}\big(R_t>0 \hspace{0.15cm}\forall t\in [T_1],\,\, R_{T_1}\in [H,2H]\big)\P( R_t > 0\hspace{0.15cm}\forall t\in \llbracket T_1, T_2\rrbracket \,\big|\, R_{T_1}=H\big).\label{Rspliteq}
\end{align}
We now fix $T_1 = 2\lfloor n^{2/3}/A^2\rfloor-1$ and $H = \lceil n^{1/3}/A\rceil$.

\begin{prop}\label{dchangemeasure}
	There exists $c>0$ such that for sufficiently large $n$ and $A$,
	\begin{equation*}
	\mathbb{P}\left(R_t>0 \hspace{0.15cm}\forall t\in [T_1], R_{T_1}\in [H,2H]\right)\geq c\frac{A}{n^{1/3}}.
	\end{equation*}
\end{prop}

\begin{prop}\label{secondstrategy}
	There exists $c>0$ such that for sufficiently large $n$ and $A$,
	\begin{equation*}
	\mathbb{P}\big(R_t>0 \hspace{0.15cm}\forall t\in \llbracket T_1, T_2\rrbracket\,\big|\,R_{T_1}=H\big)\geq \frac{c}{A^{3/2}}e^{-\frac{A^3}{8}+\frac{\lambda A^2}{2}-\frac{\lambda^2 A}{2}}.
	\end{equation*}
\end{prop}

We will prove Proposition \ref{dchangemeasure} in Section \ref{strategy1sec} and Proposition \ref{secondstrategy} in Section \ref{strategy2sec}. For now we show how these results can be used to complete the proof of the lower bounds in Theorem \ref{mainthm}.

\begin{proof}[Proof of lower bounds in Theorem \ref{mainthm}]
By Lemmas \ref{first} and \ref{rub1},
\[\mathbb{P}\left(|C(v)|>T_2\right)\geq \mathbb{P}\left(R_t>0 \hspace{0.2cm}\forall t\in [T_2] \right)-o\left(A^{-1/2}n^{-1/3}e^{-\frac{A^3}{8}+\frac{\lambda A^2}{2}-\frac{\lambda^2A}{2}}\right).\]
In light of \eqref{Rspliteq}, it then follows from Propositions \ref{dchangemeasure} and \ref{secondstrategy} that
\begin{equation}\label{lbparta}
\mathbb{P}\left(|C(v)|>T_2\right)\geq \frac{c}{A^{1/2}n^{1/3}}e^{-\frac{A^3}{8}+\frac{\lambda A^2}{2}-\frac{\lambda^2A}{2}}.
\end{equation}
This concludes the proof of the lower bound in part $(a)$ of Theorem \ref{mainthm}. In order to prove the lower bound in part $(b)$, we will need to use the fact that for any $\N_0$-valued random variable $X$,
	\begin{equation}\label{simpleineq}
	\mathbb{P}(X\geq 1)\geq \frac{\mathbb{E}[X]^2}{\mathbb{E}[X^2]}.
	\end{equation}
This can be proved by applying the Cauchy-Schwarz inequality to $X\ind_{\{X\ge 1\}}$.

To proceed with the proof of the lower bound in part $(b)$, let us denote by $X = \sum_{i=1}^{n}\mathbbm{1}_{\{|C(i)|\in [T_2,2T_2]\}}$ the number of components of size between $T_2$ and $2T_2$. Observe that $X\geq 1 $ implies $|\mathcal{C}_{\max}|\geq T_2$. Therefore using \eqref{simpleineq} we obtain
\begin{equation}\label{ratiox}
\mathbb{P}\left(|\mathcal{C}_{\max}|\geq T_2\right)\geq \frac{\mathbb{E}[X]^2}{\mathbb{E}[X^2]}.
\end{equation}
For the numerator, we have
\begin{equation}\label{EXtoP}
\mathbb{E}[X]^2=n^2\mathbb{P}\left(|C(1)|\in [T_2,2T_2]\right)^2.
\end{equation}
Next we bound the denominator from above. Given vertices $i,j\in [n]$, write $i\leftrightarrow j$ if there exists a path of opens edges between $i$ and $j$. Then we can write
\begin{equation}\label{secondmoment}
\mathbb{E}[X^2]\leq n\mathbb{P}\left(|C(1)|\in [T_2,2T_2]\right)+S_1+S_2,
\end{equation}
where 
\begin{equation*}
S_1 = \mathbb{E}\bigg[\sum_{i=1}^{n}\sum_{j\neq i}\mathbbm{1}_{\{|C(i)|\in [T_2,2T_2] \}}\mathbbm{1}_{\{|C(j)|\in [T_2,2T_2] \}}\mathbbm{1}_{\{i \nleftrightarrow j \}}\bigg]
\end{equation*}
and 
\begin{equation*}
S_2 = \mathbb{E}\bigg[\sum_{i=1}^{n}\sum_{j\neq i}\mathbbm{1}_{\{|C(i)|\in [T_2,2T_2] \}}\mathbbm{1}_{\{|C(j)|\in [T_2,2T_2] \}}\mathbbm{1}_{\{i \leftrightarrow j \}}\bigg].
\end{equation*}
For $S_1$ we have
\begin{align*}
S_1&\leq n^2\sum_{k=T_2}^{2T_2}\mathbb{P}\Big(|C(1)|=k, 1\nleftrightarrow 2\Big)\mathbb{P}\Big(|C(2)|\in [T_2,2T_2] \,\Big|\, |C(1)|=k, 1\nleftrightarrow 2\Big)\\
&\leq n^2 \mathbb{P}\left(|C(1)|\in [T_2,2T_2] \right) \mathbb{P}\left(|C(2)|\ge T_2\right).
\end{align*}
For $S_2$ we have
\begin{equation*}
\begin{split}
S_2=\mathbb{E}\bigg[\sum_{i=1}^{n}\sum_{k=T_2}^{2T_2}\mathbbm{1}_{\{|C(i)|=k\}}\sum_{j\neq i}\mathbbm{1}_{\{j\in C(i)\}}\bigg]&\leq \mathbb{E}\bigg[\sum_{i=1}^{n}\sum_{k=T_2}^{2T_2}\mathbbm{1}_{\{|C(i)|=k\}}k\bigg]\\
&\leq 2T_2 n\mathbb{P}\left(|C(1)|\in [T_2,2T_2]\right).
\end{split}
\end{equation*}
Returning to $(\ref{secondmoment})$ and recalling that $T_2 = \lceil An^{2/3}\rceil$, we see that
\[\mathbb{E}[X^2]\leq n^2 \mathbb{P}\left(|C(1)|\in [T_2,2T_2] \right)\mathbb{P}\left(|C(2)|\ge T_2\right) + 3An^{5/3}\mathbb{P}\left(|C(1)|\in [T_2,2T_2]\right).\]
By the upper bound in part $(a)$ of Theorem \ref{mainthm},
\[\mathbb{P}\left(|C(2)|\ge T_2 \right) \le \frac{c_2}{A^{1/2}n^{1/3}}\]
and therefore
\[\mathbb{E}[X^2]\leq cAn^{5/3}\mathbb{P}\left(|C(1)|\in [T_2,2T_2]\right).\]
Substituting this and \eqref{EXtoP} into \eqref{ratiox} and then applying \eqref{lbparta}, we obtain
\begin{align*}
\mathbb{P}\left(|\mathcal{C}_{\max}|\geq \lceil An^{2/3} \rceil\right)&\geq\frac{n^{2}\mathbb{P}\left(|C(1)|\in [T_2,2T_2]\right)^2}{cAn^{5/3}\mathbb{P}\left(|C(1)|\in [T_2,2T_2]\right)}\\
&=c\frac{n^{1/3}}{A}\mathbb{P}\left(|C(1)|\in [T_2,2T_2]\right)\\
&\geq \frac{c}{A^{3/2}}e^{-\frac{A^3}{8}+\frac{\lambda A^2}{2}-\frac{\lambda^2A}{2}},
\end{align*}
as required. This completes the proof of Theorem 1.2, subject to the proofs of Lemmas \ref{first} and \ref{rub1} and Propositions \ref{dchangemeasure} and \ref{secondstrategy}.
\end{proof}

\subsection{Rearranging Bernoullis and applying the ballot theorem: proof of Proposition \ref{dchangemeasure}}\label{strategy1sec}
We first introduce a technical result which will be used to transform $(R_t)_{t\in [T_1]}$ into a process with i.i.d. increments. 
	\begin{lem}\label{sum}
		Suppose that $N\in\N$, and that $L\in[N]$ is odd. Let $(I_j^{i})_{i,j\ge 1}$ be i.i.d.~non-negative random variables and set $X_i=\sum_{j=1}^{N-i}I_j^{i}$ for $i=1,\ldots,L$. Then there exist i.i.d.~random variables $(\tilde I_j^i)_{i,j\ge1}$ with the same distribution as $I^i_j$ such that if we set $\tilde X_i=\sum_{j=1}^{N-(L+1)/2}\tilde I_j^{i}$ then
		\begin{itemize}
			\item $\sum_{i=1}^{t} \tilde X_i\leq \sum_{i=1}^{t}X_i$ for all $1\leq t\leq L$;
			\item $\sum_{i=1}^{L} \tilde X_i=\sum_{i=1}^{L}X_i$.
		\end{itemize}
	\end{lem}

The reader can think of the $I_j^i$ as Bernoulli$(p)$-distributed, so that $X_i\sim Bin(N-i,p)$ and $\tilde X_i\sim Bin(N-(L+1)/2,p)$. The idea behind the proof is that $X_1$ has more summands than $X_L$, so if we transfer some of the summands from $X_1$ to $X_L$, we do not change the value of $\sum_{i=1}^L X_i$ but we decrease $X_1$. Then we move on to $X_2$, and transfer some of its summands to $X_{L-1}$, which decreases $\sum_{i=1}^2 X_i$ without changing $\sum_{i=1}^L X_i$; and so on. We postpone the details until Section \ref{4142sec}.

Before we can proceed with the proof of Proposition \ref{dchangemeasure}, we need one more tool. We can use Lemma \ref{sum} to transform $(R_t)_{t\in [T_1]}$ into a process with i.i.d.~increments, but in order to apply the generalised ballot theorem, Theorem \ref{genballot}, we need our increments also to have mean zero and for their distribution not to depend on $n$. The following lemma is slightly more general than we will need.

\begin{lem}\label{turnmeanzero}
		Take $n\in\N$, $h_n\ge 0$, $a_n\in(-1,\infty)$ satisfying $na_n\in\mathbb Z$, $b_n\in(-1,n-1)$ and $t_n\in\N$. Suppose that $M_t = 1 + \sum_{i=1}^t (W_i-1)$ where the $W_i$ are independent $\Bin(n(1+a_n),(1+b_n)/n)$ random variables. Let $\mu_n = (1+a_n)(1+b_n)$. Then
		\begin{multline*}
		\P\big(M_t>0\,\,\,\,\forall t\in [t_n], \, M_{t_n}\in[h_n,2h_n]\big)\\
		\ge (\mu_n\wedge 1)^{2h_n} \mu_n^{t_n-1} e^{(1-\mu_n)t_n}\P\big(\hat M_t > 0 \,\,\,\,\forall t\in[t_n],\, \hat M_{t_n}\in[h_n,2h_n]\big) - \frac{t_n}{n}(1+a_n)(1+b_n)^2
		\end{multline*}
		where $\hat M_t= 1+\sum_{i=1}^{t}(\hat W_i-1)$, and $(\hat W_i)_{i=1}^{t_n}$ is a sequence of independent Poisson random variables with mean one.
	\end{lem}
	
We delay the proof, which uses a fairly standard Poisson approximation for the binomial distribution and then a simple change of measure to remove the drift, until Section \ref{4648sec} and proceed with the proof of Proposition \ref{dchangemeasure}.

\begin{proof}[Proof of Proposition \ref{dchangemeasure}]
	As previously mentioned, we want to bound
	\[\mathbb{P}\left(R_t>0 \hspace{0.15cm}\forall t\in [T_1], R_{T_1}\in [H,2H]\right)\]
	by means of the generalised ballot theorem, Theorem \ref{genballot}. To this end, we first need to turn the process $(R_t)_{t\in [T_1]}$ into a random walk with i.i.d. steps having mean zero. In order to obtain identically distributed steps we will make use of Lemma \ref{sum}.

	Recall that $H=\lceil n^{1/3}/A\rceil$ and $R_t=1+\sum_{i=1}^{t}(\delta_i-1)$, where each $\delta_i$ is the sum of $n-\lfloor n^{2/5} \rfloor -i$ i.i.d.~$Ber(p)$ random variables. It follows from Lemma \ref{sum}, with $N=n-\lfloor n^{2/5}\rfloor$ and $L=T_1$, that there exists a sequence $(\tilde \delta_i)_{i\in [T_1]}$ of i.i.d~random variables with $\tilde \delta_i\sim \Bin(n-\lfloor n^{2/5}\rfloor - (T_1+1)/2,p)$ for which, setting $\tilde R_t = 1+\sum_{i=1}^{t}(\tilde \delta_i-1)$, we obtain
	\begin{equation}\label{dcm1}
	\mathbb{P}\left(R_t>0 \hspace{0.15cm}\forall t\in [T_1],\, R_{T_1}\in [H,2H]\right) \geq \mathbb{P}\left(\tilde R_t>0 \hspace{0.15cm}\forall t\in [T_1],\, \tilde R_{T_1}\in [H,2H]\right).
	\end{equation}
	
	In order to evaluate the probabilities appearing in the above sum by means of the generalised ballot theorem, we still have to turn $(\tilde R_t)_{t\in [T_1]}$ into a process whose increments have mean zero. We do this by applying Lemma \ref{turnmeanzero} with $h_n=H$, $t_n=T_1=2\lfloor n^{2/3}/A^2\rfloor-1$, $a_n = -\lfloor n^{2/5}\rfloor/n - (T_1+1)/(2n)$ and $b_n = \lambda/n^{1/3}$. Since $|\lambda|\le A/3$, it is easy to see that there exists a constant $c>0$ (not depending on $n$) such that
	\[(\mu_n\wedge 1)^{2h_n} \ge c.\]
	Also, using the inequality $1+x\ge e^{x-x^2}$ valid for $x\ge -1/2$, for sufficiently large $n$ we have
	\[\mu_n^{t_n-1} e^{(1-\mu_n)t_n} = \mu_n^{-1}(1+(\mu_n-1))^{t_n} e^{(1-\mu_n)t_n} \ge \mu_n^{-1} e^{-(\mu_n-1)^2 t_n} \ge c\]
	for some constant $c>0$. Finally, since
	\[\frac{t_n}{n}(1+a_n)(1+b_n)^2 \asymp \frac{t_n}{n} \asymp \frac{1}{A^2 n^{1/3}},\]
	from Lemma \ref{turnmeanzero} we obtain that
	\begin{multline}\label{dcm2}
	\mathbb{P}\left(\tilde R_t>0 \hspace{0.15cm}\forall t\in [T_1],\, \tilde R_{T_1}\in [H,2H]\right) \\
	\ge c\mathbb{P}\left(\hat R_t>0 \hspace{0.15cm}\forall t\in [T_1],\, \hat R_{T_1}\in [H,2H]\right) - \frac{C}{A^2 n^{1/3}}
	\end{multline}
	for some constants $c>0$ and $C<\infty$, where $\hat R_t = 1+\sum_{i=1}^t (\hat \delta_i-1)$ and $(\hat\delta_i)_{i=1}^{T_1}$ is a sequence of independent Poisson random variables with parameter $1$.
	
	We are now in a position to apply Theorem \ref{genballot}. Recalling that $H = \lceil n^{1/3}/A\rceil$ and $T_1 = 2\lfloor n^{2/3}/A^2\rfloor-1$, for all $k\in[H-1,2H-1]$ we have $k\leq 2H=O(\sqrt{T_1})$. We can therefore conclude from Theorem \ref{genballot} that
	\begin{align*}
	&\mathbb{P}\left(\hat R_t>0 \hspace{0.15cm}\forall t\in [T_1],\, \hat R_{T_1}\in [H,2H]\right) \\
	&\ge \mathbb{P}\left(\hat R_t-1>0 \hspace{0.15cm}\forall t\in [T_1],\, \hat R_{T_1}-1\in [H-1,2H-1]\right) \\
	&\ge c\sum_{k=H-1}^{2H-1} \frac{k+1}{T_1^{3/2}}
	\end{align*}
	which is of order $An^{-1/3}$. Substituting this bound into \eqref{dcm2} gives
	\[\mathbb{P}\left(\tilde R_t>0 \hspace{0.15cm}\forall t\in [T_1],\, \tilde R_{T_1}\in [H,2H]\right) \ge \frac{cA}{n^{1/3}} - \frac{C}{A^2 n^{1/3}}.\]
	Taking $A$ sufficiently large that the first term dominates, and then recalling \eqref{dcm1}, gives the result.
	\end{proof}

\subsection{Adding independent binomials and approximating with Brownian motion: proof of Proposition \ref{secondstrategy}}\label{strategy2sec}

Recall that $R_t = 1+\sum_{i=1}^t (\delta_i-1)$ where $(\delta_i)_{i=1}^{T_2}$ is a sequence of independent $\Bin(n-K-i,p)$ random variables. Recall also that $H=\lceil n^{1/3}/A\rceil$, $K=\lfloor n^{2/5}\rfloor$, $T_1 = 2\lfloor n^{2/3}/A^2\rfloor -1$ and $T_2 = \lceil An^{2/3}\rceil$. Throughout this section we write $T = T_2-T_1$.

Our first task in this section is to replace $R_t$ with a sum of i.i.d.~random variables. We do this by adding an independent $\Bin(K+i,p)$ random variable to $\delta_i$ for each $i$, and checking that the sum of these additional random variables cannot be too large using Lemma \ref{Bol2}.

\begin{lem}\label{RtoSlem}
For $t\in[0,\infty)$, define
\[g(t) = -\frac{n^{1/3}}{2A} + \frac{9t}{A^2n^{1/3}} + \frac{pt^2}{2}.\]
Then there exists $c>0$ such that for all large $n$,
\[\P(R_t>0\;\; \forall t\in \llbracket T_1, T_2\rrbracket \,|\, R_{T_1}=H ) \ge \P\big(S_t > g(t)\;\; \forall t\in\llbracket 1,T\rrbracket\big) - T e^{-cAn^{1/6}}\]
where $S_t = \sum_{i=1}^t \Delta_i$, and $(\Delta_i)_{i=1}^{T}$ is a sequence of independent $\Bin(n,p)$ random variables.
\end{lem}

The proof of Lemma \ref{RtoSlem} is in Section \ref{4142sec}. Next, in order to apply a Brownian approximation to our random walk, we would like the step distribution not to depend on $n$.

	\begin{lem}\label{StoShatlem}
	For $t\in[0,\infty)$, define
	\[\gamma(t) = -\frac{n^{1/3}}{4A} + \frac{9t}{A^2n^{1/3}} + \frac{pt^2}{2}.\]
	Let $g$ and $(S_t)_{t=0}^T$ be as in Lemma \ref{RtoSlem}. Then there exist constants $c,C\in(0,\infty)$ such that
	\begin{align*}
	&\P\big(S_t > g(t)\;\; \forall t\in\llbracket 1,T\rrbracket\big)\\
	&\ge c e^{\lambda A^2/2 - \lambda^2 A/2}\P\big(\hat S_t > \gamma(t)\;\;\forall t\in\llbracket 1,T\rrbracket,\, \hat S_{T}\le\gamma(T)+\tfrac{3n^{1/3}}{8A}\big) - C\exp\Big(- \frac{n^{1/3}}{4A}\Big),
	\end{align*}
	where $\hat S_t = \sum_{i=1}^t \hat\Delta_i$ and $(\hat\Delta_i)_{i=1}^{T}$ is a sequence of independent Poisson random variables of parameter $1$.
	\end{lem}

	Again we delay the proof, which is similar to the proof of Lemma \ref{turnmeanzero}, until Section \ref{4648sec} and proceed with the proof of Proposition \ref{secondstrategy}. As mentioned above, our strategy is to approximate the random walk $(\hat S_t)_{t=0}^{T}$ appearing in Lemma \ref{StoShatlem} with Brownian motion. We will use an accurate bound of Koml\'os, Major and Tusn\'ady due to its ease of application, although we will not really need the additional precision gained over the earlier result of Strassen \cite[Theorem 1.5]{strassen:ASsums}. The following rephrasing of the original theorem is from \cite{chatterjee:strong_embeddings}.
	
	\begin{thm}[Koml\'os, Major, Tusn\'ady \cite{komlos_major_tusnady:approximation_partial_sums}]\label{embedding}
		Let $(\xi_i)_{i\geq 1}$ be a sequence of i.i.d. random variables with $\mathbb{E}[\xi_1]=0$ and $\mathbb{E}[\xi_1^{2}]=1$. Suppose that there exists $\theta>0$ such that $\mathbb{E}\left[e^{\theta |\xi_1|}\right]<\infty$. For each $k\in\{0\}\cup\N$, let $U_k = \sum_{i=1}^k \xi_i$. Then for every $N\in \mathbb{N}$ it is possible to construct a version of $(U_k)_{k=0}^N$ and a standard Brownian Motion $(B_s)_{s\in[0,N]}$ on the same probability space such that, for every $x\geq 0$,
		\begin{equation*}
		\mathbb{P}\left(\max_{k\leq N}\left|U_k-B_k\right|>M\log N + x\right)\leq C e^{-c x}
		\end{equation*}
		where $M$, $C$ and $c>0$ do not depend on $N$.
	\end{thm}

	Applying this with $N=T$ and $U_k = \hat S_k$, we immediately obtain the following corollary.
	
\begin{cor}\label{KMTcor}
Suppose that $\hat S_t = \sum_{i=1}^t \hat{\Delta}_i$, where $(\hat \Delta_i)_{i=1}^{T}$ is a sequence of independent Poisson random variables of parameter $1$, and that $(B_s)_{s\ge 0}$ is a standard Brownian motion. There exist constants $c,C\in(0,\infty)$ such that for any $x_n\ge 0$ and any function $\gamma:[0,\infty)\to\mathbb{R}$,
\begin{multline*}
\P\big(\hat S_t > \gamma(t)\;\;\forall t\in\llbracket 1,T\rrbracket,\, \hat S_{T}\le\gamma(T)+\tfrac{3n^{1/3}}{8A}\big)\\
\ge \P\big(B_s > \gamma(s) + M\log T + x_n\;\;\forall s\in[0,T],\, B_{T}\le\gamma(T)+\tfrac{3n^{1/3}}{8A} - M\log T - x_n\big) - C e^{-c x_n}.
\end{multline*}
\end{cor}

We have now reduced our task to bounding the probability that a Brownian motion remains above a curve up to time $T$, and is not too far above the curve at time $T$.

\begin{prop}\label{Brownianprop}
There exists a constant $c>0$ such that for any $x_n$ satisfying $A^3\ll x_n \ll n^{1/3}/A$, any constant $M$ (not depending on $n$) and $\gamma$ as in Lemma \ref{StoShatlem}, for large $n$,
\[\P\big(B_s > \gamma(s) + M\log T + x_n\;\;\forall s\in[0,T],\, B_{T}\le\gamma(T)+\tfrac{3n^{1/3}}{8A} - M\log T - x\big) \ge \frac{c}{A^{3/2}}e^{-A^3/8}.\]
\end{prop}

The proof of Proposition \ref{Brownianprop} involves considering two time intervals, $[0,T/2]$ and $[T/2,T]$, and approximating $\gamma(T)$ by a straight line on each of these intervals. We carry out the details in Section \ref{Browniansec}.

We now have all the ingredients to prove Proposition \ref{secondstrategy} and therefore Theorem \ref{mainthm}.

\begin{proof}[Proof of Proposition \ref{secondstrategy}]
We simply combine Lemmas \ref{RtoSlem} and \ref{StoShatlem}, Corollary \ref{KMTcor} and Proposition \ref{Brownianprop}.
\end{proof}

\subsection{Proofs of Lemmas \ref{first}, \ref{rub1}, \ref{sum} and \ref{RtoSlem}: creating i.i.d.~sequences}\label{4142sec}

We first prove Lemma \ref{first}, which replaces $\eta_i$, the number of unseen vertices that become active at the $i$th step of the exploration process, with an independent Binomial random variable that does not depend on the history of the exploration process.

\begin{proof}[Proof of Lemma \ref{first}]
	From the description of the exploration process provided at the beginning of section 3, recall that $u_t$ is the vertex that is explored at step $t$. Let us denote by $\mathcal{A}^*_t$ the set of unseen vertices that become active at step $t-1$ of the process (with $\mathcal A^*_1 = \{u_1\}$), and let $\mathcal A_t = \bigcup_{i=0}^t \mathcal A^*_i$, the set of all active or explored vertices after step $t-1$. Also, write $X^t_v$ for the indicator that $u_t$ is a neighbour of vertex $v$.
	
	For each $t=1,2,\ldots,n-K$, if $|\mathcal A_t| < K+t$ then let $\mathcal B^*_t$ be any subset of the vertices $[n]$ such that
	\begin{itemize}
	\item $\mathcal A^*_t \subset \mathcal B^*_t$;
	\item $\mathcal B^*_t \cap \mathcal A_{t-1} = \emptyset$;
	\item $\big|\mathcal B^*_t \cup \mathcal A_{t-1}\big| = K+t$.
	\end{itemize}
	If $|\mathcal A_t| \ge K+t$ then let $\mathcal B^*_t = \mathcal A^*_t$. Then let
	\[r_t = \big| \mathcal B^*_t \cup \mathcal A_{t-1}\big| - K - t \ge 0.\]
	Take a sequence $\hat X_1^t, \hat X_2^t,\ldots$ of independent Bernoulli random variables of parameter $p$, also independent of everything else.
	
	Note that
	\[\eta_t = \sum_{v\not\in\mathcal A_t} X^t_v\]
	and define
	\[\delta_t = \sum_{v\not\in\mathcal B^*_t\cup\mathcal A_{t-1}}\hspace{-3mm} X^t_v + \sum_{i=1}^{r_t} \hat X^t_j.\]
	Then, since
	\[\big|\big(\mathcal B^*_t\cup\mathcal A_{t-1}\big)^c\big| + r_t = n - K - t,\]
	and the random variables $\{X^i_v : v\in \mathcal A_{i-1}^c\}$ are independent and independent of $\{X^j_v : v\in \mathcal A_{j-1}^c\}$ for any $j\neq i$, we see that $(\delta_t)_{t=1}^{n-K}$ is a sequence of independent random variables such that $\delta_t \sim \Bin(n-K-t,p)$. 
	
	We also observe that if $|\mathcal A_t| < K+t$, then $|\mathcal B^*_t \cup \mathcal A_{t-1}| = K+t$ and so $r_t=0$. Since we also have $\mathcal A^*_t \subset \mathcal B^*_t$, we see that if $|\mathcal A_t| < K+t$ then $\eta_t \ge \delta_t$. Thus
	\begin{align*}
	\mathbb{P}\left(|C(v)|>T_2\right) &= \P\Big(1+\sum_{i=1}^{t}(\eta_i - 1)>0 \hspace{0.2cm} \forall t\in [T_2]\Big)\\
	&\ge \P\Big(1+\sum_{i=1}^{t}(\eta_i - 1)>0 \text{ and } |\mathcal A_t| < K+t \hspace{0.2cm} \forall t\in [T_2]\Big)\\
	&\ge \P\Big(1+\sum_{i=1}^{t}(\delta_i - 1)>0 \text{ and } |\mathcal A_t| < K+t \hspace{0.2cm} \forall t\in [T_2]\Big)\\
	&\ge \P\Big(1+\sum_{i=1}^{t}(\delta_i - 1)>0 \hspace{0.2cm} \forall t\in [T_2]\Big) - \P\Big(\exists t\in [T_2] : |\mathcal A_t| \ge K+t\Big)
	\end{align*}
	Since $|\mathcal A_t| = Y_t - t$, the result follows.
\end{proof}

Next we prove Lemma \ref{rub1}, which ensures that the probability that the number of active vertices becomes too large is small.

\begin{proof}[Proof of Lemma \ref{rub1}]
	A union bound gives
	\begin{equation}\label{rubrub}
	\P\left(\exists i\leq T_2:Y_i\geq  \lfloor n^{2/5}\rfloor \right) \leq \sum_{i=1}^{T_2}\P\left(Y_i\geq \lfloor n^{2/5}\rfloor\right) \leq \sum_{i=1}^{T_2}\P\bigg(1+\sum_{j=1}^{i}(\zeta_i-1)\geq \lfloor n^{2/5}\rfloor \bigg),
	\end{equation}
	where $\zeta_i\overset{i.i.d.}{\sim}\Bin(n,p)$. Denoting by $B_{N,q}$ a binomial random variable with parameters $N$ and $q$, by Lemma \ref{Bol2} we see that for $i\in [T_2]$,
	\begin{align}\label{oneterm}
	\nonumber\P\bigg(1+\sum_{j=1}^{i}(\zeta_i-1)\geq \lfloor n^{2/5}\rfloor\bigg)&=\P\left(B_{in,p}\geq inp-i\lambda n^{-1/3}+\lfloor n^{2/5}\rfloor-1\right)\\
	&\leq \P\left(B_{in,p}\geq inp+  n^{2/5}\left(1-c\frac{A|\lambda|}{n^{1/15}}\right)\right).
	\end{align}
	Now, $A=o(n^{1/30})$ and $|\lambda|\leq A/3$, so $A|\lambda|=o\left(n^{1/15}\right)$ and hence for large enough $n$ we obtain 
	\begin{align}\label{xmx}
	(\ref{oneterm})\leq \P\left(B_{in,p}\geq inp+  n^{2/5}/2 \right)\leq  \exp\left\{-\frac{n^{4/5}/4}{2inp+\frac{1}{3}n^{2/5}}\right\}.
	\end{align}
	Since $i\leq T_2$ we see that $(\ref{xmx})\leq e^{-cn^{2/15}/A}$ for some positive constant $c>0$. Finally, since $A,\lambda=o(n^{1/30})$ as $n\rightarrow \infty$ we conclude that
	\[ e^{-cn^{2/15}/A}=o\left(A^{-1/2}n^{-1/3}e^{-\frac{A^3}{8}+\frac{\lambda A^2}{2}-\frac{\lambda^2A}{2}}\right).\qedhere\]
\end{proof}

Lemma \ref{sum} involves rearranging Bernoulli random variables to produce an i.i.d.~sequence.

\begin{proof}[Proof of Lemma \ref{sum}]
		Recall the convention that the empty sum is zero. By hypothesis,
		\begin{equation}
		X_i=\sum_{j=1}^{N-i}I_j^{i},
		\end{equation}
		where $(I_j^{i})_{i,j\geq 1}$ are i.i.d.~non-negative random variables. Let $\ell = (L+1)/2$; recall that $L$ is odd, so $\ell\in \N$. Define
		\begin{equation*}
		\tilde{X}_i = \left\{ \begin{aligned}
		& \sum_{j=1}^{N-\ell}I_j^{i}, &&  1\leq i\leq \ell\\
		& \sum_{j=1}^{N-i}I_j^{i} + \sum_{j=N-\ell+1}^{N-(L+1-i)}I_j^{L+1-i}, && \ell<i\leq L.\\
		\end{aligned}
		\right.
		\end{equation*}
		Observe that $(\tilde{X}_i)_{i=1}^{L}$ is a sequence of i.i.d. random variables and $\tilde{X}_i\overset{d}{=}\sum_{j=1}^{N-\ell}I_j^{1}$, $1\leq i\leq L$. Next we claim that  
		\begin{equation}\label{claim}
		\sum_{i=1}^{t}\tilde{X}_i\leq \sum_{i=1}^{t}X_i
		\end{equation} 
		for all $1\leq t\leq L$. To see this, observe first that when $1\leq t\leq \ell$ we have that
		\begin{align*}
		\sum_{i=1}^{t}\tilde{X}_i=\sum_{i=1}^{t}\sum_{j=1}^{N-\ell}I_j^{i}\leq \sum_{i=1}^{t}\sum_{j=1}^{N-i}I_j^{i}= \sum_{i=1}^{t}X_i.
		\end{align*}
		Next, for $\ell<t\leq L$ we have that
		\[\sum_{i=1}^{t}\tilde{X}_i =	\sum_{i=1}^{\ell}\tilde{X}_i +\sum_{i=\ell+1}^{t}\tilde{X}_i =\sum_{i=1}^{\ell}\sum_{j=1}^{N-\ell}I_j^{i}+\sum_{i=\ell+1}^{t}\sum_{j=1}^{N-i}I_j^{i}+\sum_{i=\ell+1}^{t}\sum_{j=N-\ell+1}^{N-(L+1-i)}I_j^{L+1-i}.\]
Making the change of variable $k=L+1-i$, we see that 
		\begin{equation*}
		\sum_{i=\ell+1}^{t}\sum_{j=N-\ell+1}^{N-(L+1-i)}I_j^{L+1-i} = \sum_{k=L+1-t}^{\ell-1}\sum_{j=N-\ell+1}^{N-k}I_j^{k} = \sum_{k=L+1-t}^{\ell}\sum_{j=N-\ell+1}^{N-k}I_j^{k},
		\end{equation*}
		where the last equality follows from the fact that the term corresponding to $k=\ell$ is zero.
		
		Therefore
		\begin{align}
		\sum_{i=1}^{t}\tilde{X}_i &= \sum_{i=1}^{\ell}\sum_{j=1}^{N-\ell}I_j^{i}+\sum_{i=\ell+1}^{t}\sum_{j=1}^{N-i}I_j^{i}+\sum_{i=L+1-t}^{\ell}\sum_{j=N-\ell+1}^{N-i}I_j^{i}\label{tofinish}\\
		 &\le \sum_{i=1}^{\ell}\sum_{j=1}^{N-\ell}I_j^{i}+\sum_{i=\ell+1}^{t}\sum_{j=1}^{N-i}I_j^{i}+\sum_{i=1}^{\ell}\sum_{j=N-\ell+1}^{N-i}I_j^{i}\nonumber\\
		 &= \sum_{i=1}^t X_i\nonumber
		 \end{align}
		as claimed. The second statement of the lemma simply follows by taking $t=L$ in (\ref{tofinish}), in which case the subsequent inequality is an equality.
\end{proof}

Lemma \ref{RtoSlem} provides an alternative way of producing i.i.d.~sequences of Binomial random variables, by adding independent binomials to the original sequence.

\begin{proof}[Proof of Lemma \ref{RtoSlem}]
Recall that $R_t=1+\sum_{i=1}^{t}(\delta_i-1)$. As we did in Section \ref{strategy2sec}, write $T=T_2-T_1$. Let $(L_i)_{i\in [T_2]}$ be a sequence of independent random variables, also independent of $(\delta_i)_{i\in [T_2]}$ and such that $L_i\sim \Bin(K+i,p)$. Then, setting 
\begin{equation}\label{St}
S_t = \sum_{i=T_1+1}^{T_1+t}(\delta_i+L_i-1),
\end{equation}
we see that $\delta_i+L_i \overset{i.i.d.}{\sim}\Bin(n,p)$. Let $\mathcal L_t = \sum_{i=T_1+1}^{T_1+t}L_i$. Then for any $f:\mathbb{N}\to\mathbb{R}$,
\begin{align}
\P(R_t>0 \;\;\forall t\in \llbracket T_1, T_2\rrbracket \,|\, R_{T_1}=H)&= \P\big(R_{T_1+t}-R_{T_1} > -H \;\;\forall t\in[T] \big)\nonumber\\
&= \P\big(S_t - \mathcal L_t > -H \;\; \forall t\in[T] \big)\nonumber\\
&\ge \P\big(S_t > \mathcal L_t - H \;\; \forall t\in[T],\, \mathcal L_t \le f(t)+H\;\;\forall t\in[T] \big)\nonumber\\
&\ge \P\big(S_t > f(t)\;\;\forall t\in[T]\big) - \P\big(\exists t\in[T] : \mathcal L_t > f(t)+H\big).\label{finsteadofg}
\end{align}

We now let $f(t) = \E[\mathcal L_t] + \frac{A^{1/2}}{n^{5/12}}(T_1+t) - H$ and aim to show that
\begin{equation}\label{RtoSsuff}
\P\big(\exists t\in[T] : \mathcal L_t > f(t)+H\big) \le T e^{-cAn^{1/6}}.
\end{equation}
	Indeed, a union bound gives
	\[\P\big(\exists t\in[T] : \mathcal L_t > f(t)+H\big) \le \sum_{t=1}^{T} \P\Big( \mathcal L_t \ge \E[\mathcal L_t] + \frac{A^{1/2}}{n^{5/12}}(T_1+t) \Big),\]
	and then applying Lemma \ref{Bol2} yields
	\begin{equation}\label{RtoSBol}
	\P\big(\exists t\in[T] : \mathcal L_t > f(t)+H\big) \le \sum_{t=1}^{T} \exp\bigg(-\frac{A(T_1+t)^2 n^{-5/6}}{2(K+T_1+1/2)tp+t^2p+\frac{2A^{1/2}(T_1+t)}{3n^{5/12}}}\bigg).
	\end{equation}
	One may easily check that for some finite constant $C$, we have
	\[2(K+T_1+1/2)tp \le CT_1 t/n \le C(T_1+t)^2/n,\]
	\[t^2p \le Ct^2/n \le C(T_1+t)^2/n\]
	and
	\[\frac{2A^{1/2}(T_1+t)}{3n^{5/12}} \le C\frac{T_1}{n}(T_1+t)\le C(T_1+t)^2/n.\]
	Thus, for some $c>0$,
	\[\exp\bigg(-\frac{A(T_1+t)^2 n^{-5/6}}{2(K+T_1+1/2)tp+t^2p+\frac{2A^{1/2}(T_1+t)}{3n^{5/12}}}\bigg) \le \exp(-cAn^{1/6})\]
	which combines with \eqref{RtoSBol} to give \eqref{RtoSsuff}.
	
	Substituting \eqref{RtoSsuff} into \eqref{finsteadofg} proves the Lemma with $f$ in place of $g$. It therefore suffices to check that $f(t)\le g(t)$ for all $t$ when $n$ is large. This holds since
	\[\mathcal L_t\sim\Bin\bigg(\sum_{i=T_1+1}^{T_1+t}(K+i),p\bigg) = \Bin\big((K+T_1)t + t(t+1)/2,p\big),\]
	and we have
	\[p(K+T_1+1/2)t \le \frac{8t}{n^{1/3}A^2},\;\; A^{1/2}n^{-5/12}T_1 \le \frac{n^{1/3}}{2A} \;\;\text{and}\;\; A^{1/2}n^{-5/12}t \le \frac{t}{n^{1/3}A^2}.\]
	These estimates show that $f(t)\le g(t)$ and complete the proof.
	\end{proof}

\subsection{Proofs of Lemmas \ref{turnmeanzero} and \ref{StoShatlem}: Poisson approximation and a change of measure}\label{4648sec}

The proof of Lemma \ref{turnmeanzero} uses two standard ingredients: a coupling between Binomial and Poisson random variables, and a change of measure to remove the drift from a random walk.

\begin{proof}[Proof of Lemma \ref{turnmeanzero}]
	Note that $\E[W_i] = (1+a_n)(1+b_n) = \mu_n$ for each $i$. By \cite[Theorem 2.10]{remco:random_graphs} we can construct a coupling between $(W_i)_{i\in\N}$ and a sequence $W_i'$ of i.i.d.~Poisson random variables with parameter $\mu_n$, such that
	\[\P(W_i\neq W_i') \le \sum_{i=1}^{n(1+a_n)} \Big(\frac{1+b_n}{n}\Big)^2 = \frac{(1+a_n)(1+b_n)^2}{n}.\]
	Let $M_t' = \sum_{i=1}^t (W_i'-1)$. Then
	\begin{multline}\label{turnmean0step1}
	\P\big(M_t>0\,\,\,\,\forall t\in[t_n], \, M_{t_n}\in[h_n,2h_n]\big)\\
	\ge \P\big(M_t'>0\,\,\,\,\forall t\in[t_n], \, M_{t_n}'\in[h_n,2h_n]\big) - \P\big(\exists i\in[t_n] : W_i\neq W_i'\big),
	\end{multline}
	and a union bound gives that
	\begin{equation}\label{turnmean0step2}
	\P\big(\exists i\in[t_n] : W_i\neq W_i'\big) \le t_n\frac{(1+a_n)(1+b_n)^2}{n}.
	\end{equation}
	
	We now seek to remove the drift from the sequence $M_t$ by using a change of measure. Define a new probability measure $\Q$ by setting, for $B\in\sigma(W_1',\ldots,W_{t_n}')$,
	\begin{equation}\label{CoM}
	\Q(B) = \E\bigg[\ind_B \prod_{i=1}^{t_n} \mu_n ^{-W_i'}\bigg] \E\Big[\mu_n^{- W_1'}\Big]^{-t_n} = \E\Big[\ind_B \mu_n^{- M_{t_n}'-t_n + 1}\Big] e^{(\mu_n-1)t_n}.
	\end{equation}
	We write $\E_\Q$ for expectation with respect to the probability measure $\Q$. It is straightforward to check that, under $\Q$, the $W_i'$ are independent Poisson random variables of parameter $1$.
	
	By the definition of $\Q$, we have
	\begin{align*}
	&\P\big(M_t'>0\,\,\forall t\in\llbracket 0,t_n\rrbracket, \, M_{t_n}'\in[h_n,2h_n]\big)\\
	&= \E_\Q\Big[\mu_n^{M'_{t_n}+t_n-1} \ind_{\{M_t'>0\,\forall t\in\llbracket 0,t_n\rrbracket, \, M_{t_n}'\in[h_n,2h_n]\}}\Big]e^{(1-\mu_n)t_n}\\
	&\ge (\mu_n\wedge 1)^{2h_n-1} \mu_n^{t_n-1}\Q\big(M_t'>0\,\,\forall t\in\llbracket 0,t_n\rrbracket, \, M_{t_n}'\in[h_n,2h_n]\big)e^{(1-\mu_n)t_n}.
	\end{align*}
	Since $(W_i')_{i=1}^{t_n}$ is a sequence of independent Poisson random variables of parameter $1$ under $\Q$, substituting this and \eqref{turnmean0step2} into \eqref{turnmean0step1} gives the result.
	\end{proof}

	The proof of Lemma \ref{StoShatlem} is similar to that of Lemma \ref{turnmeanzero}, but we will need to delve deeper into the details of the coupling between the Binomial and Poisson random variables.

	\begin{proof}[Proof of Lemma \ref{StoShatlem}]
	We follow almost the same proof as Lemma \ref{turnmeanzero}, noting that $\E[\Delta_i] = np$ for each $i$. By \cite[Theorem 2.10]{remco:random_graphs} we can couple $(\Delta_i)_{i=1}^T$ with a sequence $(\Delta_i')_{i=1}^T$ of i.i.d.~Poisson random variables with parameter $np$. Write $S_t' = \sum_{i=1}^t (\Delta_i'-1)$. Then
	\begin{multline}\label{Shatstep1}
	\P\big(S_t>g(t)\,\,\,\,\forall t\in\llbracket 0,T\rrbracket\big)\\
	\ge \P\big(S_t'>g(t) + \tfrac{n^{1/3}}{4A}\,\,\,\,\forall t\in\llbracket 0,T\rrbracket\big) - \P\big(\max_{t\le T}|S_t-S_t'|> \tfrac{n^{1/3}}{4A}\big).
	\end{multline}
	To estimate the last probability, we see that
	\begin{equation}\label{Doobmaxbd}
	\P\Big(\max_{t\le T} |S_t - S_t'| > \frac{n^{1/3}}{4A}\Big) \le \P\bigg(\sum_{i=1}^T |\Delta_i-\Delta'_i| > \frac{n^{1/3}}{4A}\bigg) \le \E[e^{|\Delta_1 - \Delta_1'|}]^T e^{-n^{1/3}/(4A)},
	\end{equation}
	where for the last inequality we used the i.i.d. property of the increments $\Delta_i-\Delta_i'$. To continue our bounds we need some more detail about the coupling of $\Delta_1$ and $\Delta'_1$ from the proof of \cite[Theorem 2.10]{remco:random_graphs}. We break $\Delta_1$ up into a sum of $n$ i.i.d.~Bernoulli random variables of parameter $p$, which we call $(\beta_j)_{j=1}^n$, and couple these with $n$ Poisson random variables $(\beta'_j)_{j=1}^n$ of parameter $p$, so that
	\[\Delta_1 = \sum_{j=1}^n \beta_{j} \hspace{4mm} \text{ and } \hspace{4mm} \Delta'_1 = \sum_{j=1}^n \beta'_{j}.\]
	The coupling is arranged so that for each $i$ and $j$,
	\begin{itemize}
	\item $\P(\beta_{j} = \beta'_{j} = 0) = 1-p$,
	\item $\P(\beta_{j} = \beta'_{j} = 1) = pe^{-p}$,
	\item $\P(\beta_{j} = 1,\,\beta'_{j}=0) = e^{-p}-(1-p)$
	and
	\item $\P(\beta_{j} = 1,\,\beta'_{j}=k) = \P(\beta'_{j}=k) = \frac{e^{-p}p^k}{k!}$ for $k\ge 2$.
	\end{itemize}
	We deduce (using the inequality $e^{-x}\le 1-x+x^2/2$, valid for all $x\ge 0$) that
	\begin{align*}
	\E[e^{|\beta_{j}-\beta'_{j}|}] &= 1-p+pe^{-p} + e(e^{-p}-(1-p)) + \sum_{k=2}^\infty e^{k-1}\frac{e^{-p}p^k}{k!}\\
	&\le 1 + e\frac{p^2}{2} + p^2 e e^{-p} \sum_{k=0}^\infty \frac{(ep)^k}{(k+2)!} \le 1 + cp^2
	\end{align*}
	for some finite constant $c$. Thus
	\[\E[e^{|\Delta_1 - \Delta_1'|}] \le (1+cp^2)^n \le \exp(cp^2 n),\]
	and substituting this into \eqref{Doobmaxbd} gives
	\begin{equation}\label{Shatstep2}
	\P\Big(\max_{t\le T} |S_t - S_t'| > \frac{n^{1/3}}{4A}\Big) \le \exp\Big(cp^2nT- \frac{n^{1/3}}{4A}\Big)\le C\exp\Big(- \frac{n^{1/3}}{4A}\Big)
	\end{equation}
	for some finite constant $C$.
	
	We now consider the first quantity on the right-hand side of \eqref{Shatstep1}, and use the same change of measure as in \eqref{CoM} with $\mu_n = pn$ to remove the drift from $S_t'$. Noting that $g(t) + \frac{n^{1/3}}{4A} = \gamma(t)$, by the definition of $\Q$, for any $\ell\ge0$,
	\begin{align}
	&\P\big(S_t'>g(t) + \tfrac{n^{1/3}}{4A}\,\,\,\,\forall t\in\llbracket 0,T\rrbracket\big)\nonumber\\
	&\ge\P\big(S_t'>\gamma(t)\,\,\,\,\forall t\in\llbracket 0,T\rrbracket,\, S_T' \le \gamma(t)+\ell\big)\nonumber\\
	&=\E_\Q\Big[(pn)^{S'_T + T}\ind_{\{S_t'>\gamma(t)\,\,\,\,\forall t\in\llbracket 0,T\rrbracket,\, S_T' \le \gamma(t)+\ell\}}\Big]e^{(1-pn)T}\nonumber\\
	&\ge (pn\wedge 1)^{\ell}(pn)^{\gamma(T)+T} e^{(1-pn)T} \Q\big(S_t'>\gamma(t)\,\,\,\,\forall t\in\llbracket 0,T\rrbracket,\, S_T' \le \gamma(t)+\ell\big).\label{Shatstep3}
	\end{align}
	Taking $\ell = \frac{3n^{1/3}}{8A}$ and recalling that $pn = 1+\lambda n^{-1/3}$ and
	\[\gamma(T) = -\frac{n^{1/3}}{4A} + \frac{9T}{A^2 n^{1/3}} + \frac{pT^2}{2} = \frac{A^2 n^{1/3}}{2} + O\Big(\frac{n^{1/3}}{A}\Big),\]
	and using that $|\lambda|\le A/3$, $|A|=o(n^{1/30})$ and $1+x\geq e^{x-x^2}$ for all $x>-1/2$ we have
	\[(pn\wedge 1)^\ell \ge c,\]
	\[(pn)^{\gamma(T)} = (1+\lambda n^{-1/3})^{A^2 n^{1/3}/2 + O(n^{1/3}/A)} \ge ce^{\lambda A^2/2}\]
	and
	\[(pn)^T e^{(1-pn)T} = ((1+\lambda n^{-1/3})e^{-\lambda n^{-1/3}})^T = e^{-\lambda^2 n^{-2/3} T/2 + O(\lambda^3 n^{-1}T)} \ge c e^{-\lambda^2 A/2}.\]
	Substituting these estimates into \eqref{Shatstep3}, we obtain
	\begin{multline*}
	\P\big(S_t'>g(t) + \tfrac{n^{1/3}}{4A}\,\,\,\,\forall t\in\llbracket 0,T\rrbracket\big)\\
	\ge c e^{\lambda A^2/2 - \lambda^2 A/2}\Q\big(S_t'>\gamma(T)\,\,\,\,\forall t\in\llbracket 0,T\rrbracket,\, S_T' \le \gamma(T)+\tfrac{3n^{1/3}}{8A}\big).
	\end{multline*}
	Since $(S_t')_{t=1}^{T}$ is a sum of independent Poisson random variables of parameter $1$ under $\Q$, substituting this and \eqref{Shatstep2} into \eqref{Shatstep1} gives the result.
\end{proof}

\subsection{The probability a Brownian motion stays above a curve: proof of Proposition \ref{Brownianprop}}\label{Browniansec}
	Recall that
		\[\gamma(s)=-\frac{n^{1/3}}{4A}+\frac{9s}{A^2n^{1/3}}+\frac{ps^2}{2}\]
	and write
	\[P_n(T) = \P\big(B_s > \gamma(s) + M\log T + x_n\;\;\forall s\in[0,T],\, B_{T}\le\gamma(T)+\tfrac{3n^{1/3}}{8A} - M\log T - x\big);\]
	our aim in this section is to bound $P_n(T)$ from below.
	
	Define
	\[\phi(s) = \gamma(s) + \frac{n^{1/3}}{8A} = -\frac{n^{1/3}}{8A} + \frac{9s}{A^2n^{1/3}} + \frac{ps^2}{2}\]
	and
	\[\psi_T = \gamma(T) + \frac{n^{1/3}}{4A} = \frac{9T}{A^2n^{1/3}} + \frac{pT^2}{2}.\]
	Note that since $x_n\ll n^{1/3}/A$, for large $n$ we have $M\log T + x_n\le n^{1/3}/(8A)$, so
	\begin{equation}\label{BrownianStep1}
	P_n(T) \ge \P\big(B_s > \phi(s) \;\;\forall s\in[0,T],\, B_{T}\le\psi_T\big).
	\end{equation}
	We approximate the curve $\phi(s)$ given above with two straight lines defined, for $s\in [0,T/2]$, by
	\[\ell_1(s) = \phi(0) + \Big(\frac{\phi(T/2)-\phi(0)}{T/2}\Big)s = -\frac{n^{1/3}}{8A} + \Big(\frac{9}{A^2 n^{1/3}}+\frac{pT}{4}\Big)s\]
	and
	\[\ell_2(s) = \phi(T/2) + \Big(\frac{\phi(T)-\phi(T/2)}{T/2}\Big)s = -\frac{n^{1/3}}{8A} + \frac{9T}{2A^2 n^{1/3}}+\frac{pT^2}{8} + \Big(\frac{9}{A^2 n^{1/3}}+\frac{3pT}{4}\Big)s.\]
	Also define
	\begin{align*}
	I &= \Big[\psi_T/2 - A^{1/2}n^{1/3},\, \psi_T/2\Big]\\
	&= \Big[\frac{9T}{2A^2n^{1/3}} + \frac{pT^2}{4} - A^{1/2}n^{1/3},\, \frac{9T}{2A^2n^{1/3}} + \frac{pT^2}{4}\Big].
	\end{align*}
	See Figure \ref{graphfig} for reference. Note that $\psi_T/2 - A^{1/2}n^{1/3} > \phi(T/2)$ when $n$ is large, so the interval $I$ falls entirely above the curve $\phi$.
	
	\begin{figure}[h]
    \def\svgwidth{160mm}
	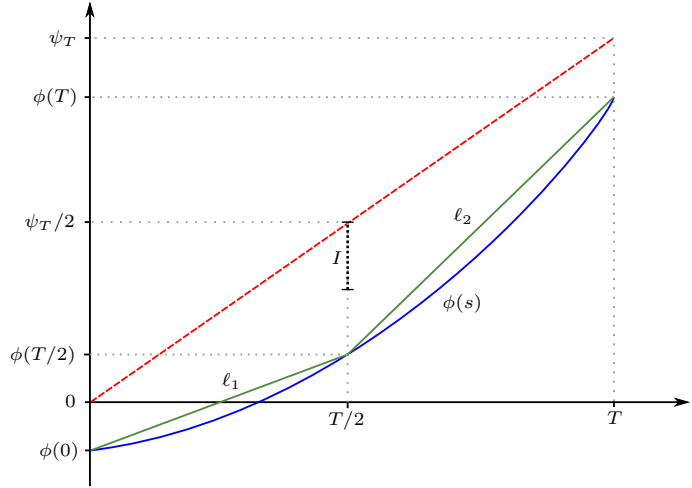
	\vspace{-7mm}
	\caption{We want our Brownian motion to stay above the blue curve, and the two green lines $\ell_1$ and $\ell_2$ show linear approximations to this curve on the two half-intervals. The dashed red line shows roughly where we expect our Brownian motion to be, given that it stays above the curve. This is a caricature of the true picture, and not to scale.}
	\label{graphfig}
	\end{figure}
	
	Since $\phi$ is convex, the linear interpolations $\ell_1$ and $\ell_2$ fall above the curve and therefore \eqref{BrownianStep1} is at least
	\[\P\big(B_s > \ell_1(s)\;\; \forall s\in[0,T/2],\, B_s > \ell_2(s-T/2) \;\;\forall s\in[T/2,T],\, B_T \le \psi_T\big).\]
	For a lower bound, we may also insist that at time $T/2$, our Brownian motion falls within the interval $I$; putting all this together, we obtain that
	\begin{multline}\label{BrownianStep2}
	P_n(T) \ge \int_I \P\big(B_s > \ell_1(s)\;\; \forall s\in[0,T/2],\, B_{T/2} \in \d w\big)\\
	\cdot \P_w\big(B_s > \ell_2(s)\;\; \forall s\in[0,T/2],\, B_{T/2} \le \psi_T\big).
	\end{multline}
	Here $\P_w$ denotes a probability measure under which our Brownian motion starts from $w$ rather than $0$.
	
	\begin{lem}\label{reflectionlem}
	For any $\mu,y\in\mathbb{R}$, $t>0$, $x>y$ and $z>y+\mu t$,
	\[\P_x(B_s > y + \mu s \;\;\forall s\le t,\, B_t \in \d z) = \frac{1}{\sqrt{2\pi t}} \exp\Big(-\frac{(z-x)^2}{2t}\Big)\big(1-e^{2(z-x-y-\mu t)y/t}\big)\d z.\]
	\end{lem}
	\begin{proof}
	We begin with an exponential change of measure to balance the drift $\mu$. Letting $(\F_s)_{s\ge 0}$ be the natural filtration of our Brownian motion, define $\mathcal P_x$, with expectation operator $\mathcal E_x$, by setting
	\[\frac{\d \mathcal P_x}{\d\P_x}\Big|_{\mathcal F_t} = e^{\mu B_t - \mu x - \mu^2 t/2}.\]
	Then under $\mathcal P_x$, $(B_s)_{s\ge 0}$ is a Brownian motion with drift $\mu$ started from $x$, and therefore
	\begin{align*}
	\P_x(B_s > y + \mu s \;\;\forall s\le t,\, B_t \in \d z) &= \mathcal E_x\big[ e^{-\mu B_t + \mu^2 t/2 + \mu x}\ind_{\{B_s > y + \mu s \;\;\forall s\le t,\, B_t \in \d z\}}\big]\\
	&= e^{-\mu z + \mu^2 t/2 + \mu x}\mathcal P_x(B_s > y + \mu s \;\;\forall s\le t,\, B_t \in \d z)\\
	&= e^{-\mu z + \mu^2 t/2 + \mu x}\P_x(B_s > y \;\;\forall s\le t,\, B_t + \mu t \in \d z).
	\end{align*}
	We now recall that, as a consequence of the reflection principle for Brownian motion, for $x>y$ and $w>y$,
	\begin{align*}
	\P_x(B_s > y \;\;\forall s\le t,\, B_t \in dw) &= \frac{1}{\sqrt{2\pi t}}\Big(\exp\Big(\frac{-(w-x)^2}{2t}\Big) - \exp\Big(-\frac{(w-x-2y)^2}{2t}\Big)\Big)\d w\\
	&= \frac{1}{\sqrt{2\pi t}}\exp\Big(\frac{-(w-x)^2}{2t}\Big)\Big(1 - \exp\Big(\frac{2(w-x-y)y}{t}\Big)\Big)\d w.
	\end{align*}
	Taking $w=z-\mu t$ and substituting into the expression above, and then simplifying, gives the result.
	\end{proof}
	
	We now use Lemma \ref{reflectionlem} to obtain a lower bound for the probability that $B_t$ stays above the line $l_1(s)$ and finishes near $w\in I$ at time $T/2$.
	
\begin{cor}\label{BrownianCor1}
For $w\in I$,
\[\P\big(B_s > \ell_1(s)\;\; \forall s\in[0,T/2],\, B_{T/2} \in \d w\big) \ge \frac{c}{\sqrt T} e^{-w^2/T}\d w.\]
\end{cor}

\begin{proof}
We apply Lemma \ref{reflectionlem} with $x=0$, $y=-n^{1/3}/(8A)$, $\mu = \frac{9}{A^2 n^{1/3}} + \frac{pT}{4}$ and $t=T/2$. With these parameters, $w>y+\mu t$ and hence Lemma \ref{reflectionlem} tells us that
\begin{multline*}
\P\big(B_s > \ell_1(s)\;\; \forall s\in[0,T/2],\, B_{T/2} \in \d w\big)\\
= \frac{1}{\sqrt{\pi T}} e^{-w^2/T} \Big(1-\exp\Big(-2\Big(w+\frac{n^{1/3}}{8A} - \frac{9T}{2A^2 n^{1/3}} - \frac{pT^2}{8}\Big)\frac{n^{1/3}}{4AT}\Big)\Big)\d w.
\end{multline*}
Since $w\in I$, we have $w\ge \frac{9T}{2A^2 n^{1/3}} + \frac{pT^2}{4} - A^{1/2}n^{1/3}$ and therefore
\[\Big(w+\frac{n^{1/3}}{8A} - \frac{9T}{2A^2 n^{1/3}} - \frac{pT^2}{8}\Big)\frac{n^{1/3}}{4AT} \ge \Big(\frac{pT^2}{8}-A^{1/2}n^{1/3}\Big)\frac{n^{1/3}}{4AT} \ge c\]
for some $c>0$, and the result follows.
\end{proof}

Next we bound from below the second probability that appears in the integral (\ref{BrownianStep2}), again by means of Lemma \ref{reflectionlem}.

\begin{cor}\label{BrownianCor2}
For $w\in I$ and $A$ sufficiently large,
\[\P_w\big(B_s > \ell_2(s)\;\; \forall s\in[0,T/2],\, B_{T/2} \le \psi_T\big) \ge \frac{c}{\sqrt{T}} \int_{\phi(T)}^{\psi_T} e^{-(z-w)^2/T} \d z.\]
\end{cor}

\begin{proof}
We now apply Lemma \ref{reflectionlem} with $y=-\frac{n^{1/3}}{8A} + \frac{9T}{2A^2 n^{1/3}} + \frac{pT^2}{8}$, $\mu = \frac{9}{A^2 n^{1/3}} + \frac{3pT}{4}$ and $t=T/2$. This tells us that
\begin{multline}\label{P2eq}
\P_w\big(B_s > \ell_2(s)\;\; \forall s\in[0,T/2],\, B_{T/2} \le \psi_T\big)\\
= \frac{1}{\sqrt{\pi T}} \int_{\phi(T)}^{\psi_T} e^{-(z-w)^2/T} \big(1-e^{4(z-w-y-\mu T/2)y/T}\big)\d z.
\end{multline}
Now, for $w\in I$ and $z\in[\phi(T),\psi_T]$, we have
\[z-w \le \frac{9T}{A^2n^{1/3}} + \frac{pT^2}{2} - \frac{9T}{2A^2n^{1/3}} - \frac{pT^2}{4} + A^{1/2}n^{1/3} = \frac{9T}{2A^2n^{1/3}} + \frac{pT^2}{4} + A^{1/2}n^{1/3}\]
and
\[y + \mu T/2 = -\frac{n^{1/3}}{8A} + \frac{9T}{A^2 n^{1/3}} + \frac{pT^2}{2}\]
so
\[z-w-y-\mu T/2 \le \frac{n^{1/3}}{8A} - \frac{9T}{2A^2 n^{1/3}} - \frac{pT^2}{4} + A^{1/2}n^{1/3} = -\frac{A^2 n^{1/3}}{4} + O(A^{1/2} n^{1/3}).\]
Also
\[\frac{y}{T} = -\frac{n^{1/3}}{8AT} + \frac{9}{2A^2 n^{1/3}} + \frac{pT}{8} = \frac{A}{8n^{1/3}} + O\Big(\frac{1}{A^2 n^{1/3}}\Big).\]
Thus the exponential term appearing at the end of \eqref{P2eq} is $e^{-A^3/8 + O(A^{3/2})}$, which is smaller than $1$ when $A$ is large, and therefore
\[\P_w\big(B_s > \ell_2(s)\;\; \forall s\in[0,T/2],\, B_{T/2} \le \psi_T\big) \ge \frac{c}{\sqrt{T}} \int_{\phi(T)}^{\psi_T} e^{-(z-w)^2/T} \d z\]
as required.
\end{proof}

\begin{proof}[Proof of Proposition \ref{Brownianprop}]
Substituting Corollaries \ref{BrownianCor1} and \ref{BrownianCor2} into \eqref{BrownianStep2}, we obtain
\[P_n(T)\ge \frac{c}{T} \int_{\phi(T)}^{\psi_T} \int_I  e^{-w^2/T - (z-w)^2/T} \d w \, \d z.\]
Using the substitutions $u = w - \psi_T/2$ and $v = z - \psi_T$, the above equals
\[\frac{c}{T} \int_{\phi(T)-\psi_T}^{0} \int_{-A^{1/2}n^{1/3}}^0  e^{-(u+\psi_T/2)^2/T - (v-u+\psi_T/2)^2/T} \d u \, \d v\]
which, after multiplying out the quadratic terms in the exponent, becomes
\[\frac{c}{T} \int_{\phi(T)-\psi_T}^{0} \int_{-A^{1/2}n^{1/3}}^0  e^{-2u^2/T + 2uv/T - \psi_T^2/(2T) - v^2/T - v\psi_T/T} \d u \, \d v.\]
Since $u,v\le 0$, we have $2uv/T\ge 0$ and therefore the integral over $u$ is at least
\[\int_{-A^{1/2}n^{1/3}}^0  e^{-2u^2/T} \d u \ge c A^{1/2}n^{1/3} \ge c\sqrt T.\]
We deduce that
\begin{equation}\label{BrownianStep3}
	P_n(T) \ge \frac{c}{\sqrt{T}} \int_{\phi(T)-\psi_T}^{0}  e^{ - \psi_T^2/(2T) - v^2/T - v\psi_T/T} \d v.
\end{equation}
Now $\phi(T)-\psi_T = -n^{1/3}/(8A)$ and
\[\frac{\psi_T}{T} = \frac{9}{A^2n^{1/3}} + \frac{pT}{2} = \frac{A}{2n^{1/3}} + O\Big(\frac{1}{A^2 n^{1/3}}\Big),\]
so the exponent on the right-hand side of \eqref{BrownianStep3} is $e^{-\psi_T^2/2T - O(1)}$; thus \eqref{BrownianStep3} becomes
\[P_n(T)\ge \frac{c}{\sqrt{T}} \cdot \frac{n^{1/3}}{8A} e^{ - \psi_T^2/(2T) } \ge \frac{c}{A^{3/2}} e^{ - \psi_T^2/(2T) }.\]
It then remains only to note that
\[\frac{\psi_T^2}{2T} = \frac{1}{2T}\Big(\frac{9T}{A^2 n^{1/3}} + \frac{pT^2}{2}\Big)^2 = \frac{81T}{2A^2 n^{1/3}} + \frac{9pT^2}{2A^2 n^{1/3}} + \frac{p^2 T^3}{8} = \frac{A^3}{8} + O(1),\]
and the proof is complete.
\end{proof}

\section*{Acknowledgements}

Both authors would like to thank Nathana\"el Berestycki for some very helpful discussions. We also thank the Royal Society for their generous funding, of a PhD scholarship for UDA and a University Research Fellowship for MR.

\bibliographystyle{plain}
\def\cprime{$'$}

\end{document}